\documentclass[11pt]{amsart}
\usepackage{amsmath,amssymb,amsthm,amscd}
\usepackage{colordvi,graphicx}
\usepackage[margin=1in]{geometry}
\usepackage{hyperref}
\usepackage[enableskew]{youngtab}
\usepackage[all,cmtip]{xy}
\usepackage{color}
\usepackage{qtree}
\usepackage{tikz}
\usepackage[normalem]{ulem}

\newtheorem{theorem}{Theorem}[section]
\newtheorem{lemma}[theorem]{Lemma}
\newtheorem{proposition}[theorem]{Proposition}
\newtheorem{corollary}[theorem]{Corollary}

\theoremstyle{definition}

\newtheorem{example}[theorem]{Example}

\theoremstyle{remark}
\newtheorem{remark}[theorem]{Remark}

\numberwithin{table}{section}
\numberwithin{figure}{section}

\def\CC{{\mathbb C}}
\def\comb{\mathrm{comb}}  \def\L{\mathcal{L}}

\newcommand{\nola}{{\rule{2pt}{0pt}}} 
\renewcommand{\S}{{\mathfrak S}}
\newcommand{\T}{{\mathcal T}} \def\bx{\mathbf{x}}
\newcommand{\C}{{\mathcal C}}

\newcommand{\ld}{\delta}
\newcommand{\md}{d}
\DeclareMathOperator{\tr}{tr} 
\DeclareMathOperator{\des}{des} 

\newcommand{\defcolor}[1]{{\color{blue}{#1}}}
\newcommand{\demph}[1]{\defcolor{{\sl #1}}}

\title{Modular Catalan Numbers}

\author{Nickolas Hein}
\address{Department of Mathematics and Computer Science, Benedictine College, Atchison, KS 66002, USA}
\email{\href{mailto:nhein@benedictine.edu}{nhein@benedictine.edu}}

\author{Jia Huang}
\address{Department of Mathematics and Statistics, University of Nebraska at Kearney, Kearney, NE 68849, USA}
\email{\href{mailto:huangj2@unk.edu}{huangj2@unk.edu}}

\thanks{The authors thank Christian Krattenthaler, Jeremy Martin, Victor Reiner, and Eric Rowland for their valuable suggestions and comments.}

\keywords{Catalan number, Tamari order, Motzkin number, pattern avoidance.}

\begin{document}

\begin{abstract}
The Catalan number $C_n$ enumerates parenthesizations of $x_0*\dotsb*x_n$ where $*$ is a binary operation.
We introduce the modular Catalan number $C_{k,n}$ to count equivalence classes of parenthesizations of $x_0*\dotsb*x_n$ when $*$ satisfies a $k$-associative law generalizing the usual associativity.
This leads to a study of restricted families of Catalan objects enumerated by $C_{k,n}$ with emphasis on binary trees, plane trees, and Dyck paths, each avoiding certain patterns.
We give closed formulas for $C_{k,n}$ with two different proofs.
For each $n\ge0$ we compute the largest size of $k$-associative equivalence classes and show that the number of classes with this size is a Catalan number.
\end{abstract}

\maketitle

\section{Introduction}

Let $X$ be a set with a binary operation $*:X^2\to X$ and $x_0,\dotsc,x_n$ be $X$-valued indeterminates.
A binary operation $*$ induces a map $X^{n+1}\to X$ given by the expression $x_0*\dotsb*x_{n}$ and a choice of an order to apply each $*$.
The expression $x_0*\dotsb*x_{n}$ alone may be ambiguous, so it might not define a map without using parentheses to record the order of operations.
The number of ways to parenthesize $x_0*\dotsb*x_{n}$ is the Catalan number $C_n=\frac{1}{n+1}{2n\choose n}$ which enumerates hundreds~\cite{EC2,StanleyCatalan} of families of other natural objects.
This Catalan number gives an upper bound for the number of ways to interpret the meaning of $x_0*\dotsb*x_{n}$.

When $*$ is associative, one has $(x_0*x_1)*x_2 = x_0*(x_1*x_2)$, and thus all parenthesizations of $x_0*\dotsb*x_n$ are equal.
We investigate a natural generalization of this case.
Let $k\ge1$ be a positive integer, and suppose $*$ is a left-to-right binary operation.
We say a (left-to-right) binary operation $*$ is \demph{$k$-associative} if
\[ (x_0*\dotsb *x_k)*x_{k+1} = x_0*(x_1*\dotsb *x_{k+1}).\]
The results in this paper are also valid for right-to-left binary operations, subject to a reflection.

One may define a $k$-associative binary operation on any ring $R$ with an element $\omega$ of multiplicative order $k$ by $a*b := \omega a + b$.
Consider the ring $R=\CC$ and the primitive $k$th root of unity $\omega = e^{2\pi i/k}$ for a concrete example.

We say two parenthesizations are \demph{$k$-equivalent} if they are equal by the $k$-associative property.
We define the \demph{($k$-)modular Catalan number $C_{k,n}$} to be the number of $k$-equivalence classes of parenthesizations of $x_0*\dotsb*x_n$. 
Since $1$-associativity is the usual associativity, we have $C_{1,n}=1$.
The first nontrivial example is $C_{2,3}=4$.
We illustrate this by listing the $C_3=5$ parenthesizations for $n=3$,
\[
   ((x_0*  x_1)* x_2)* x_3,\quad 
    (x_0*  x_1)*(x_2*  x_3),\quad
    (x_0* (x_1*  x_2))*x_3,\quad
     x_0*((x_1*  x_2)* x_3),\quad
     x_0* (x_1*( x_2*  x_3)),
  \]
and observing the first and fourth parenthesizations are $2$-equivalent.
 
Modular Catalan numbers appear elsewhere for small fixed values of $k$.
The On-Line Encyclopedia of Integer Sequences (OEIS)~\cite{OEIS} sequence A005773 coincides with $\{C_{3,n}\}$.
This sequence counts directed $n$-ominoes in standard position~\cite{omino}, $n$-digit base three numbers whose digits sum to $n$, 
permutations of $[n]:=\{1,2,\ldots,n\}$ avoiding $1$-$3$-$2$ and $123$-$4$~\cite{Mansour}, minimax elements in the affine Weyl group of the Lie algebra $\mathfrak{so}_{2n+1}$ (or $\mathfrak{sp}_{2n}$)~\cite{Panyushev}, and other objects as well.
Rowland~\cite{Rowland} studied the case $k=4$, and his point of view of pattern avoidance in binary trees is relevant to our investigation.
We found no results for $k\ge5$ in the literature.

\begin{table}[h]
\[
\footnotesize\begin{tabular}{c|lllllllllllllll|c}
$n$ & 0 & 1 & 2 & 3 & 4 & 5 & 6 & 7 & 8 & 9 & 10 & 11 & 12 & 13 & 14 \\
\hline
$C_{1,n}$ & 1 & {\color{blue}1} & {\color{red}1} & 1  & 1 & 1  & 1 & 1 & 1  & 1 & 1  & 1 & 1 & 1 & 1 & A000012 \\
$C_{2,n}$ & 1 & 1 & {\color{blue}2} & {\color{red}4} & 8 & 16 & 32 & 64 & 128 & 256 & 512 & 1024 & 2048 & 4096 & 8192 & A011782 \\
$C_{3,n}$ & 1 & 1 & 2 & {\color{blue}5} & {\color{red}13} & 35 & 96 & 267 & 750 & 2123 & 6046 & 17303 & 49721 & 143365 & 414584 & A005773 \\
$C_{4,n}$ & 1 & 1 & 2 & 5 & {\color{blue}14} & {\color{red}41} & 124 & 384 & 1210 & 3865 & 12482 & 40677 & 133572 & 441468  & 1467296
 & A159772 \\
$C_{5,n}$ & 1 & 1 & 2 & 5 & 14 & {\color{blue}42} & {\color{red}131} & 420 & 1375 & 4576 & 15431 & 52603 & 180957 & 627340 & 2189430
 & \defcolor{new} \\
$C_{6,n}$ & 1 & 1 & 2 & 5 & 14 & 42 & {\color{blue}132} & {\color{red}428} & 1420 & 4796 & 16432 & 56966 & 199444 & 704146 & 2504000
 & \defcolor{new} \\
$C_{7,n}$ & 1 & 1 & 2 & 5 & 14 & 42 & 132 & {\color{blue}429} & {\color{red}1429} & 4851 & 16718 & 58331 & 205632 & 731272 & 2620176 & \defcolor{new} \\
$C_{8,n}$ & 1 & 1 & 2 & 5 & 14 & 42 & 132 & 429 & {\color{blue}1430} & {\color{red}4861} & 16784 & 58695 & 207452 & 739840 & 2658936
 & \defcolor{new} \\
$C_n$ & 1 & 1 & 2 & 5 & 14 & 42 & 132 & 429 & 1430 & 4862 & 16796 & 58786 & 208012 & 742900 & 2674440 & A000108 
\end{tabular}\]
\caption{Modular Catalan number $C_{k,n}$ for $n\le14$ and $k\le8$.}
\label{tab:Cnk}
\end{table}

A computation gives the data in Table~\ref{tab:Cnk} above. We highlight the entries ${\color{blue}{C_{k,k}}}$ and ${\color{red}{C_{k,k+1}}}$ which satisfy the relationships ${\color{blue}{C_{k,k}}}=C_{k}$ and ${\color{red}{C_{k,k+1}}}=C_{k}-1$, and we list OEIS sequences that coincide with $\{C_{k,n}\}$.

Our definition of $C_{k,n}$ using $k$-associative binary operations and parenthesizations provides a new perspective for these numbers.
It is natural and works for all $k\ge1$.
It is based on basic concepts in algebra and has connections to many interesting combinatorial objects as well, as we will observe in later sections.
Our main result in this paper is Theorem~\ref{thm1} below, which gives two closed formulas for the modular Catalan numbers.
This generalizes previously known formulas for $C_{k,n}$ with $1\le k\le 4$.
The first formula uses the evaluations of \demph{monomial symmetric functions} $m_\lambda$, which can be rewritten as certain multinomial coefficients, for partitions $\lambda$ inside a $(k-1)\times n$ rectangle.
The second formula is a simple summation with alternating signs.

\begin{theorem}\label{thm1}
For $k,n\ge1$ we have 
\[ C_{k,n} = \sum_{\substack{\lambda\subseteq(k-1)^n \\ |\lambda|<n }} \frac{n-|\lambda|}{n} m_\lambda(1^n) 
= \sum_{0\le j\le (n-1)/k} \frac{(-1)^j}{n} {n\choose j} {2n-jk\choose n+1}. \] 
\end{theorem}

To establish Theorem~\ref{thm1}, we first study the connection of parenthesizations to binary trees and plane trees in Section~\ref{sec:trees}, which is summarized below.

Let $\defcolor{\T}$ denote the set of all binary trees.
We define a (left-to-right) binary operation $\defcolor{\wedge}:\T\times \T \to \T$ where $\defcolor{s\wedge t}$ is the binary tree whose root has left and right subtrees $s$ and $t$, respectively.
There is a natural bijection between the set of parenthesizations of $x_0*\cdots *x_n$ and the set $\defcolor{\T_n}$ of binary trees with $n$ internal nodes (i.e., with $n+1$ leaves) by replacing each $x_i$ by a leaf labeled $i$ and replacing each $*$ by $\wedge$.
We define the \demph{$k$-associative order} on $\T_n$ by
\[(t_0\wedge t_1\wedge \dotsb\wedge t_k)\wedge t_{k+1} < t_0\wedge(t_1\wedge\dotsb\wedge t_{k+1}) \]
where each $t_i$ is a binary tree.
If $k\mid k'$ then the $k'$-associative order is weaker than the $k$-associative order.
In particular, any $k$-associative order is weaker than the $1$-associative order, which is called the \demph{Tamari order}.
Under the Tamari order, $\T_n$ becomes a lattice, called the \demph{Tamari lattice}, which has been widely investigated (see, e.g.,~\cite{TamariPhys,FishelNelson,Reading}) since its introduction by Tamari~\cite{Tamari}.
We define the \demph{$k$-components} of $\T_n$ to be the connected components of $\T_n$ under the $k$-associative order, which correspond to $k$-equivalence classes of parenthesizations of $x_0*\cdots *x_n$. 
The maximal and minimal elements of a $k$-component are called $k$-maximal and $k$-minimal, respectively.

We also translate the $k$-associative order to plane trees, as they are in natural bijection with binary trees.
For our purposes, it is sometimes more convenient to deal with plane trees than binary trees.
 
We prove that each $k$-component of $\T_{n}$ contains a unique minimal element.  
Consequently, the modular Catalan number $C_{k,n}$ enumerates the $k$-minimal elements of $\T_n$.
This is closely related the \demph{generalized Motzkin number} $M_{k,n}$ which counts $k$-maximal elements of $\T_{n}$.
We show that the $k$-minimal and $k$-maximal elements of $\T_n$ may be described using subtree avoidance in binary trees or degree constraints in plane trees.

Remarkably, Theorem~\ref{thm:largest} and Corollary~\ref{cor:largest} assert the number of largest $k$-components of $\T_{n}$ is the Catalan number $C_m$, where $m$ is the least positive integer congruent to $n$ modulo $k$.

Next, in Section~\ref{sec:objects}, we describe several restricted families of Catalan objects enumerated by $C_{k,n}$ and $M_{k,n}$, using bijections among them (see Proposition~\ref{prop:M counts} and Proposition~\ref{prop:C counts}).
This implies that the generalized Motzkin numbers and modular Catalan numbers are interlaced,
\[C_{1,n}\le M_{1,n} \le C_{2,n}\le M_{2,n} \le \cdots \,.\]

Section~\ref{sec:formula} includes a proof for Theorem~\ref{thm1} using generating functions and Lagrange inversion, as well as other related results.
We show the generating functions of $C_{k,n}$ and $M_{k,n}$ 
satisfy polynomial equations and are closely related to each other, as seen in Proposition~\ref{prop:gen}.
We give the first formula of Theorem~\ref{thm1} in Corollary~\ref{cor:FormulaPositive} and the second in Theorem~\ref{thm:FormulaAlt}.

Corollary~\ref{cor:FormulaPositive} and Theorem~\ref{thm:FormulaAlt} give formulas for $M_{k,n}$, analogous to those for $C_{k,n}$ of Theorem~\ref{thm1}.
These formulas for $M_{k,n}$ may be derived from work of Tak\'acs~\cite{Takacs} on plane trees with degree constraints.
One may specialize these formulas to compute the \demph{Motzkin number $M_n:=M_{2,n}$} (see OEIS A001006), which counts permutations avoiding certain patterns~\cite{perm,Mansour}, standard Young tableaux of height at most three~\cite{YoungTab}, minimax elements in the affine Weyl group of the Lie algebra $\mathfrak{sl}_{n+1}$~\cite{Panyushev}, and many other objects~\cite[Ex. 6.38]{EC2}. 
For $k=3,\dotsc,7$, the sequences $\{M_{k,n}\}$ coincide with the OEIS sequences A036765,$\,\dotsc,\,$A036769, respectively.

Our generating function approach to study $k$-minimal and $k$-maximal elements of $\T_n$ is also used to prove Proposition~\ref{prop:FormulaLargest}, which shows the size of the largest $k$-components of $\T_n$ equals
\[ \sum_{0\le j\le n/k} \frac{n-jk}{n}{n+j-1 \choose j} .\]

In Section~\ref{sec:Dyck} we use certain rotations of Dyck paths to give a more direct proof for Theorem~\ref{thm1}, with negative signs from sign-reversing involutions, and a similar proof for the above formula for the size of the largest $k$-components of $\T_n$.

It is well-known that the Catalan number $C_n$ can be refined to the \emph{Narayana number} 
\[ N_{n,r} := \frac 1n {n\choose r}{n\choose r-1} \]
which enumerates plane trees with $n+1$ total nodes, of which $r$ are internal, Dyck paths of length $2n$ with $r$ peaks, and many other objects (see, e.g.,~\cite[Ch. 2]{Petersen}).
We provide similar refinements of $C_{k,n}$ and $M_{k,n}$ in Section~\ref{sec:refinement}.

Finally, we provide remarks and questions in Section~\ref{sec:questions}.

\section{Parenthesizations and trees}\label{sec:trees}

In this section we study $k$-equivalence classes of parenthesizations via binary trees and plane trees.
A \demph{plane tree} is a rooted tree such that the children of each node are linearly ordered from left to right.
The \demph{degree} of a node is the number of its children.
Degree-zero nodes are \demph{leaves}, and all others are \demph{internal nodes}.
A tree $t$ whose edges and nodes are contained in $t$ is a \demph{subtree} of $t$.
If $v$ is a node of $t$, then the \demph{(maximal) subtree rooted at $v$} is the subtree of $t$ whose nodes are $v$ and all descendants of $v$.
The \demph{$i$th subtree} of $v$ is the subtree rooted at the $i$th child of $v$.
A \demph{binary tree} is a plane tree whose nodes have degree either zero or two.
We consider binary tree and plane tree to be different objects in this paper.

\subsection{Binary trees}\label{sec:binary}
Denote by $\defcolor{\T_n}$ the set of binary trees with $n+1$ leaves.
Let $\defcolor{s\wedge t}$ be the binary tree whose root has left and right subtrees $s$ and $t$.
There is a natural bijection between the set of parenthesizations of $x_0 * \dotsb * x_n$ and $\T_n$ given by replacing each $x_i$ by a leaf labeled $i$ and replacing each $*$ by $\wedge$.

\begin{example}\label{example:correspondence}
We list all binary trees in $\T_3$ and their corresponding parenthesizations in Figure~\ref{fig:T2P}.
\begin{figure}[h]
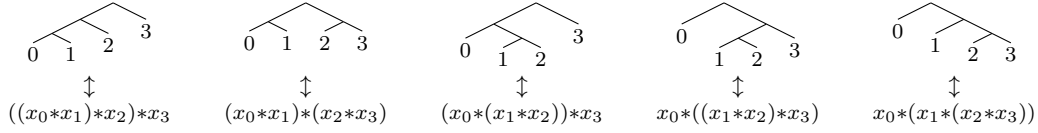

\[\scriptsize\begin{array}{ccccc}
\Tree [.  [. [. 0 1 ] 2 ] 3 ] &
\Tree [.  [. 0 1 ] [. 2 3 ] ] &
\Tree [.  [. 0 [. 1 2 ] ] 3 ] &
\Tree [. 0 [. [. 1 2 ] 3 ] ] &
\Tree [. 0  [. 1 [. 2 3 ] ] ] \\
\updownarrow & \updownarrow & \updownarrow & \updownarrow & \updownarrow \\
\rule{5pt}{0pt}((x_0 {*} x_1) {*} x_2) {*} x_3\rule{5pt}{0pt} &
\rule{5pt}{0pt}(x_0 {*} x_1) {*} (x_2 {*} x_3)\rule{5pt}{0pt} &
\rule{5pt}{0pt}(x_0 {*} (x_1{*} x_2)) {*} x_3 \rule{5pt}{0pt} &
\rule{5pt}{0pt}x_0 {*} ((x_1 {*} x_2) {*} x_3)\rule{5pt}{0pt} &
\rule{5pt}{0pt}x_0 {*} (x_1 {*} (x_2 {*} x_3))\rule{5pt}{0pt}
\end{array}\]
\caption{Correspondence between binary trees and parenthesizations}\label{fig:T2P}
\end{figure}
\end{example}
The \demph{left depth of node $v$} of $t$ is the number of left steps in the path from the root down to $v$.
Write $\demph{\ld_i(t)}$ for the left depth of leaf $i$ in binary tree $t$ and call $\defcolor{\ld(t)}:=(\ld_0(t),\ldots,\ld_{n}(t))$ the \demph{left depth of tree $t$}.
The five binary trees in Example~\ref{example:correspondence} have left depths $(3,2,1,0)$, $(2,1,1,0)$, $(2,2,1,0)$, $(1,2,1,0)$, and $(1,1,1,0)$. 

We construct a set \demph{$\mathcal D_n$} by setting $\mathcal D_0:=\{(0)\}$ and recursively defining $\mathcal D_n$ for $n\ge1$ as follows:
\[ \mathcal D_{n}:=\bigcup_{1\le i\le n} \{(a_0+1,\ldots,a_{i-1}+1,b_0,\ldots,b_{n-i})\ |\ a\in \mathcal D_{i-1}\,,\ b\in \mathcal D_{n-i}\}\,.\]
By induction on $n$, we have a surjection $\delta: \T_n\twoheadrightarrow \mathcal D_n$ by $t\mapsto \delta(t)$.
To see $\delta$ is injective, let $(\delta_0,\ldots,\delta_n)\in\mathcal D_n$. 
Then $(\delta_0-1,\ldots,\delta_{i-1}-1) \in \mathcal D_{i-1}$ and $(\delta_i,\ldots,\delta_n)\in \mathcal D_{n-i}$ for some $i\in[n]$. 
Since $(\delta_0-1,\ldots,\delta_{i-1}-1) = \delta(s)$ for some $s\in\T_{i-1}$, it follows that $i$ is the smallest positive integer such that $\delta_{i-1}=1$.
This implies $\delta$ is injective (and the union in the definition of $\mathcal D_n$ is disjoint).

\begin{example}\label{example:ring}
Let $R$ be a ring with an element $\omega$ of multiplicative order $k$. 
Define $a*b:=\omega a+b$ for all $a,b\in R$.
This gives a $k$-associative binary operation on $R$.
A binary tree $t$ with left depth $\ld(t)=(\ld_0,\ldots,\ld_{n})$ determines a parenthesization of $x_0*\dotsb *x_n$ which may be written $\sum_{0\le i\le n} \omega^{\ld_i} x_i$.
Thus the $k$-equivalence relation on parenthesizations of $x_0*\dotsb*x_n$ is the same as the congruence relation modulo $k$ on the left depths of binary trees in $\T_n$.
We will show that the same result holds for any $k$-associative binary operation $*$.
\end{example}

We take the operation $\wedge$ on trees to be a left-to-right operation so that $r\wedge s\wedge t:=(r\wedge s)\wedge t$.
Let $t_0,\dotsc,t_{k+1}$ be binary trees, and suppose $t\in \T_n$ has subtree $s:=(t_0\wedge t_1\wedge \dotsb\wedge t_k)\wedge t_{k+1}$ rooted at node $v$.
Replacing $s$ by $s':=t_0\wedge(t_1\wedge\dotsb\wedge t_{k+1})$ gives another binary tree $t'\in \T_n$.
We call the operation $t\mapsto t'$ a \demph{right $k$-rotation at $v$} and denote it by $t\xrightarrow{k} t'$.
We call the inverse operation a \demph{left $k$-rotation at $v$}.
If $t\in \T_n$ may be obtained by applying finitely many left $k$-rotations to $t'\in \T_n$, then we say \demph{$t \le t'$}.
The induced partial order on $\T_n$ is the \demph{$k$-associative order}.
The set $\T_n$ endowed with the $1$-associative order is the well-known \demph{Tamari lattice}.
Connected components of $\T_n$ under the $k$-associative order are called \demph{$k$-components}.
\begin{example}
The left poset in Figure~\ref{fig:Tamari} shows the Tamari order on $\T_4$, and the right poset shows the $2$-associative order on $\T_4$ with eight $2$-components having cardinality $1,1,1,1,2,2,3$, and $3$ respectively.
\end{example}

\begin{figure}[h]
\[
 \includegraphics[width=.45\textwidth]{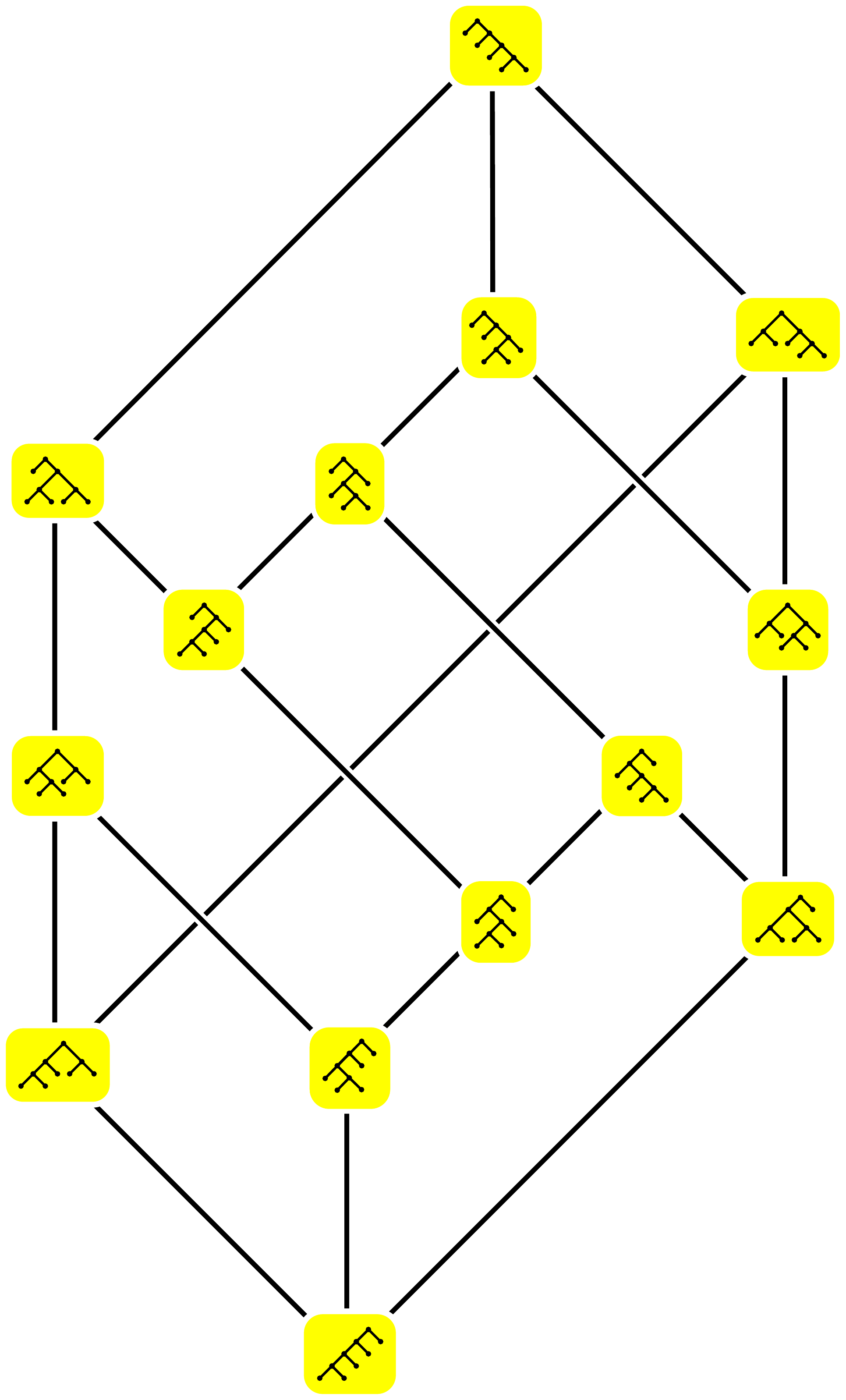}
 \rule{.08\textwidth}{0pt}
 \includegraphics[width=.47\textwidth]{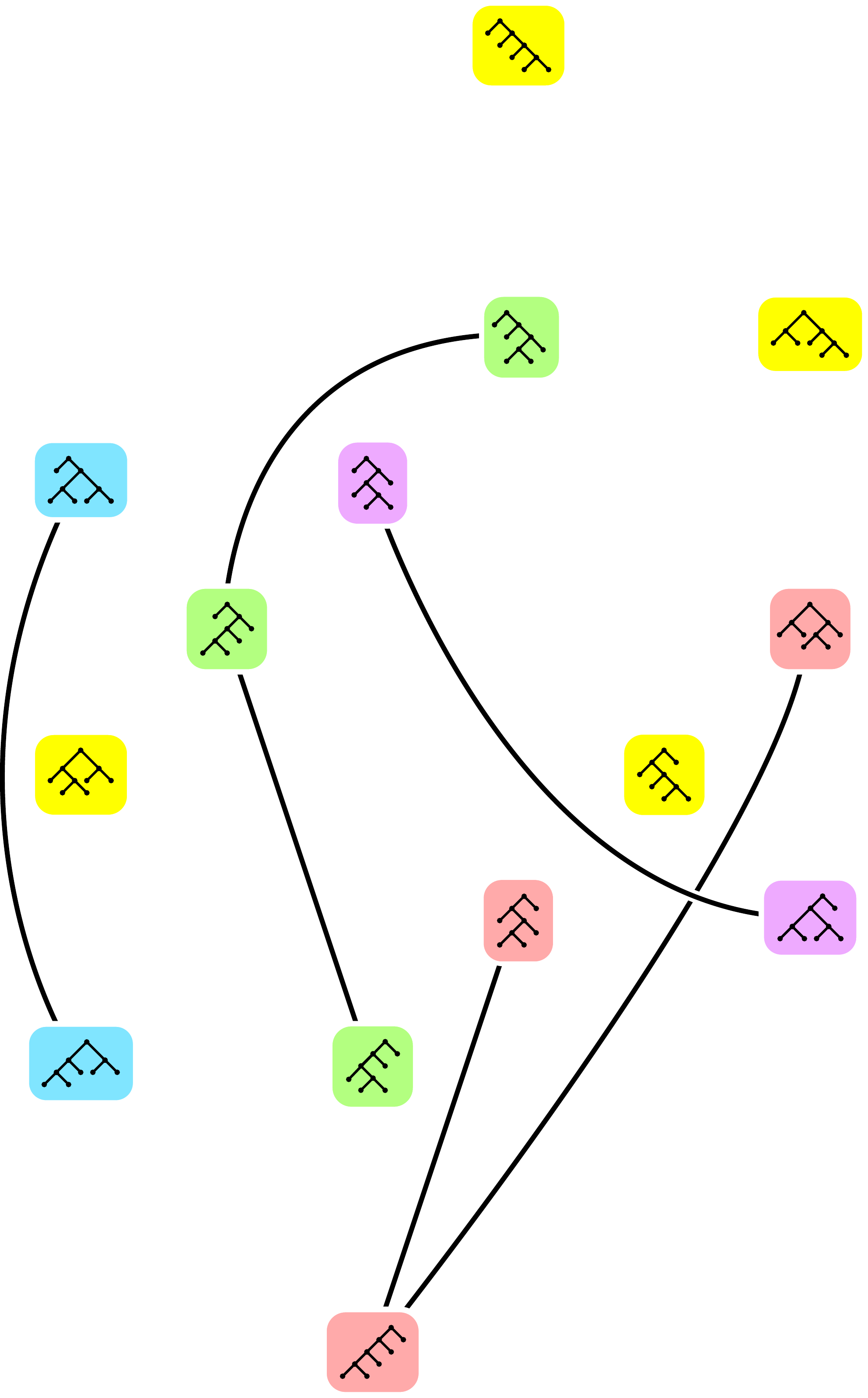}
\]
\caption{Tamari order and $2$-associative order on $\T_4$}\label{fig:Tamari}
\end{figure}
Two parenthesizations of $x_0*\cdots*x_n$ are $k$-equivalent if and only if their corresponding binary trees are \demph{$k$-equivalent}, which means they are in the same $k$-component of $\T_{n}$.

\begin{proposition}\label{P:number_classes_is_Catalan}
The modular Catalan number $C_{k,n}$ enumerates the $k$-components of $\T_{n}$.
\end{proposition}
A right (respectively left) $k$-rotation is a composition of $k$ right (respectively left) $1$-rotations, and hence corresponds to an upward (respectively downward) chain of length $k$ in the Tamari lattice. 
We illustrate this in Figure~\ref{fig:3-rotation}, decomposing a $3$-rotation into three $1$-rotations.
\begin{figure}[h]
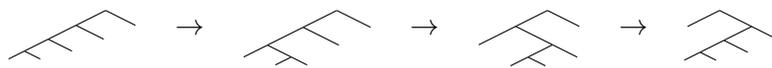

\[
\Tree[. [. [. [. {\nola} {\nola} ] {\nola} ] {\nola} ] {\nola} ] \quad\rightarrow\quad
\Tree[. [. [. {\nola} [. {\nola} {\nola} ] ] {\nola} ] {\nola} ]
\quad\rightarrow\quad
\Tree[. [. {\nola}  [. [. {\nola}  !\qsetw{.2in} {\nola} ] {\nola} ] ]  !\qsetw{.5in} {\nola} ]
\quad\rightarrow\quad
\Tree[. {\nola} !\qsetw{-.2in} [. [. [. {\nola} {\nola} ] {\nola} ] {\nola} ] ] 
\]
\caption{Decomposition of a right 3-rotation into three right 1-rotations}\label{fig:3-rotation}
\end{figure}
Thus the $k$-associative order is weaker than the Tamari order.
We generalize this below.

\begin{proposition}
If $k=pk'$ for a positive integer $p$, then a right (respectively, left) $k$-rotation may be decomposed into a sequence of $p$ right (respectively, left) $k'$-rotations.
Consequently, the $k$-associative order is weaker than the $k'$-associative order.
\end{proposition}

\begin{proof}
 We prove this for right rotations, and the result for left rotations follows.
 Assume $k=pk'$ for some positive integer $p$ and induct on $p$.
 The base case $k=k'$ is trivial.
 For $k=(p+1)k'$, we decompose a right $k$-rotation $t_0\wedge t_1\wedge \cdots\wedge t_{k+1}\overset{k}{\longrightarrow} t_0\wedge(t_1\wedge\cdots\wedge t_{k+1})$ into a right $pk'$-rotation
 \[ t_0\wedge t_1 \wedge\cdots\wedge t_{pk'+1}\wedge t_{pk'+2}\wedge\cdots\wedge t_{pk'+k'+1}\overset{pk'}{\longrightarrow} t_0\wedge(t_1\wedge\cdots\wedge t_{pk'+1})\wedge t_{pk'+2}\wedge\cdots\wedge t_{pk'+k'+1} \]
followed by a right $k'$-rotation
\[t_0\wedge(t_1\wedge\cdots\wedge t_{pk'+1})\wedge t_{pk'+2}\wedge\cdots\wedge t_{pk'+k'+1}  \overset{k'}{\longrightarrow} t_0\wedge(t_1\wedge\cdots\wedge t_{pk'+1}\wedge t_{pk'+2}\wedge\cdots\wedge t_{pk'+k'+1}). \]
Applying the inductive assumption to the above right $pk'$-rotation completes the proof.
\end{proof}

\subsection{Plane trees}

Contracting each northeast-southwest edge of a binary tree gives a plane tree.
This defines a bijection from binary trees with $n+1$ leaves to plane trees with $n+1$ (total) nodes.
It is essentially the inverse of the Knuth transform, which sends a plane tree to its left-child right-sibling representation.
See, e.g.,~\cite{AnalyticComb,Knuth}.
We give an example of our bijection in Figure~\ref{fig:BijectionTrees}.

\begin{figure}[h]
{\scriptsize \[
\Tree[. [. [. 0 !\qsetw{0in} [. [. 1 [. 2 3 ] ] 4 ] ] 5  ] [. [. 6 !\qsetw{.5in} 7 ] 8 ] ]
\qquad\qquad \raisebox{-30pt}{$\overset{}{\longleftrightarrow}$} \qquad\qquad
 \begin{picture}(70,10) 
  \put(10,-45){\includegraphics[scale=.32]{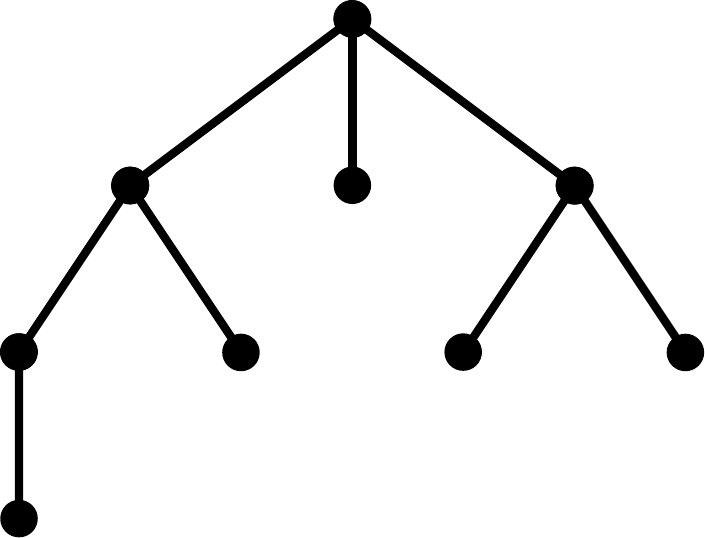}}
  \put(40.5,6){0}
  \put(17,-10){1}
  \put(5,-28){2}
  \put(10,-53){3}
  \put(30,-37){4}
  \put(40.5,-22){5}
  \put(64,-10){6}
  \put(50,-37){7}
  \put(72,-37){8}
 \end{picture}
\] }
\caption{A bijection between binary trees and plane trees}\label{fig:BijectionTrees}
\end{figure}
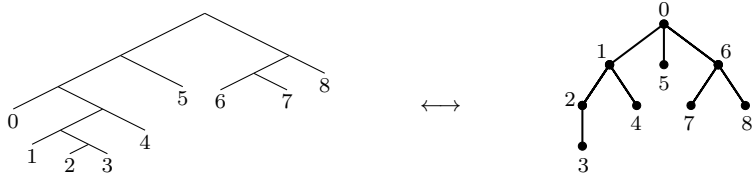

Let $T$ be the plane tree corresponding to some $t\in\T_n$.
The mapping $t\mapsto T$ associates leaf $i$ of $t$ to node $v_i$ of $T$ for $0\le i\le n$.
As the leaves labeled $0,\dotsc,n$ are ordered left-to-right, the nodes $v_0,\dotsc,v_n$ are ordered according to the \demph{pre-order}.
This order may also be obtained by first labeling the root of $T$ and then labeling the nodes of the subtrees of the root recursively in the same way, proceeding from the leftmost subtree to the rightmost one.
We define the \demph{multi-degree} of $T$ to be the degree vector $\md(T):=(\md_0(T),\ldots,\md_n(T))$, where $\md_i(T)$ is the degree of $v_i$ for each $i$.

\begin{proposition}\label{prop:d-del-corresp}
Let $t\in\T_n$ be a binary tree with left depth $\ld(t)=(\ld_0,\ldots,\ld_n)$ and $T$ be a plane tree with multi-degree $\md(T)=(\md_0,\ldots,\md_n)$.
If $t$ corresponds to $T$ via the above bijection, then 
\[ \ld_i = \md_0+\cdots+\md_i-i, \qquad \forall i\in\{0,1,\ldots,n\}\,. \]
\end{proposition}

\begin{proof}
 As before, we label leaves $0,\dotsc,n$ of $t$ left-to-right, and we label the corresponding nodes $v_0,\dotsc,v_n$ of $T$ according to the pre-order.
 The equality $\ld_0=\md_0$ is apparent from the construction of the bijection.
It remains to show that $\ld_i=\ld_{i-1}+\md_i-1$ for all $i\in[n]$.

First assume $\md_{i-1}=0$. Then $v_{i-1}$ of $T$ has no children, so leaf $i{-}1$ of $t$ is a right child.
 Thus the shortest path from leaf $i{-}1$ to leaf $i$ in $t$ is one step northwest, one step northeast, one step southeast, and $p\ge 0$ steps southwest.
 Since the first and third steps of this do not count towards left depth, $\ld_i=\ld_{i-1}-1+p$.
 Using pre-order, we see node $v_i$ of $T$ is an adjacent right sibling of node $v_{i-1}$ as $v_{i-1}$ has no children.
 Consequently, the degree $\md_i$ of $v_i$ is the number of southwest steps $p$ given above.
 Substitution gives $\ld_i=\ld_{i-1}+\md_i-1$.
  
 Now suppose $\md_{i-1}\neq 0$.
 Then, leaf $i{-}1$ of $t$ is a left child, so the shortest path from leaf $i{-}1$ to leaf $i$ in $t$ is one step northeast, one step southeast, and $p\ge 0$ steps southwest.
 Since the second step does not count towards left depth, we (again) have $\ld_i=\ld_{i-1}-1+p$ and $p=\md_i$.
\end{proof}

The following result is stated in~\cite[(19) and (20)]{Takacs} without proof.
 
\begin{proposition}\label{prop:degree}
The map $T\mapsto\md(T)$ is a bijection from plane trees with $n{+}1$ nodes to sequences $(\md_0,\ldots,\md_n)$ of $n+1$ nonnegative integers satisfying $d_0+\cdots+d_n=n$ and $\md_0 + \cdots + \md_{i-1} \ge i $ for all $i\in[n]$.
\end{proposition}

\begin{proof}
Let $T$ be a plane tree with multi-degree $\md(T)=(\md_0,\ldots,\md_n)$.
Counting non-root nodes of $T$ gives $d_0+\cdots+d_n=n$.
Suppose $t$ is the corresponding binary tree with left depth $\ld(t)=(\ld_0,\ldots,\ld_n)$.
Since $\delta_i\ge1$ unless $i=n$, Proposition~\ref{prop:d-del-corresp} implies 
\[ \md_0+\cdots+\md_{i-1}=\delta_{i-1}+i-1\ge i,\quad\forall i\in[n].\]

To show $d$ is a bijection, it suffices to construct its inverse.
Let $(\md_0,\ldots,\md_n)$ be a sequence of nonnegative integers satisfying $d_0+\cdots+d_n=n$ and $\md_0 + \cdots + \md_{i-1} \ge i $ for all $i\in[n]$.
We construct the unique plane tree $T$ with $\md(T)=(\md_0,\ldots,\md_n)$.
Let tree $T_0$ have a single node, and mark that node.
For $i=1,\ldots,n$, we construct $T_i$ by adding $\md_{i-1}$ children to the most recently marked node of $T_{i-1}$ and then marking the next node in $T_{i}$ according to pre-order.
This is possible at each step, since $\md_0+\cdots+\md_{i-1} \ge i$ for all $i\in[n]$.
The tree $T_n$ constructed in the final step is the unique plane tree with $\md(T)=(\md_0,\ldots,\md_n)$.
\end{proof}

We say two plane trees are \demph{$k$-equivalent} if their corresponding binary trees are $k$-equivalent.
We also define an \demph{up (respectively down) $k$-slide} on a plane tree $T$ to be the operation induced by a left (respectively right) $k$-rotation on the binary tree corresponding to $T$.

We describe a general up $k$-slide in more detail.
Suppose $T$ has nodes $v_0,\ldots,v_n$, and let $T_1,\ldots,T_\ell$ be the subtrees (ordered left-to-right) of a node $v_j$ with parent $v_i$.
If $\ell\ge k$, we may apply an up $k$-slide at $v_j$, giving another plane tree $T'$ with $n+1$ nodes by moving $T_{\ell-k+1},\ldots,T_{\ell}$ to new positions directly below $v_i$ and to the immediate right of $v_j$.
Although $T_{\ell-k+1},\ldots,T_\ell$ are moved, their positions in $T$ and $T'$ are the same according to the pre-order.
Thus the relation between the multi-degrees $\md(T') = (\md'_0,\ldots,\md'_n)$ and $\md(T) = (\md_0,\ldots,\md_n)$ is 
\begin{equation}\label{eq:degree} 
\md'_i = \md_i+k,\quad \md'_j = \md_j-k, \quad\text{and}\quad \md'_h = \md_h \quad \forall h\notin\{i,j\} \quad(i<j).
\end{equation}
We give an example of an up $2$-slide in Figure~\ref{fig:2-slide}, which corresponds to the following change in multi-degree: $(1,{\color{blue}3},0,{\color{blue}3},0,{\color{red}2,0,0,0},0)
\mapsto
(1,{\color{blue}5},0,{\color{blue}1},0,{\color{red}2,0,0,0},0)$.
\begin{figure}[h]
\[
\includegraphics[scale=.32]{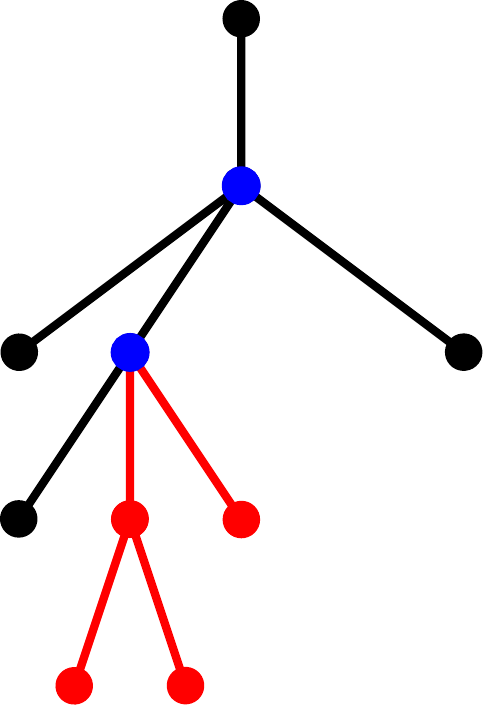}
\qquad
\raisebox{35pt}{$\longmapsto$}
\qquad
\raisebox{14.5pt}{\includegraphics[scale=.32]{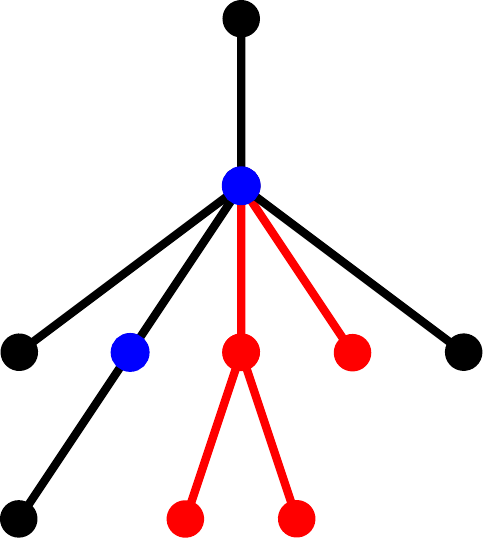}}
\]
\caption{Up 2-slide}\label{fig:2-slide}
\end{figure}

Now suppose the roots of subtrees $T_1,\dotsc,T_\ell$ are $u_1,\dotsc,u_\ell$, ordered left-to-right.
If $\ell\ge k+1$ we may apply a down $k$-slide at $v_j$ for any choice of $h\in[\ell-k]$.
This gives a plane tree $T''$ by moving $T_{h+1},\ldots,T_{h+k}$ down so that they are rooted at $u_h$ while preserving pre-order.

The $k$-slides generate a partial order on plane trees, which is called the \demph{$k$-associative order} as it is equivalent to the $k$-associative order on binary trees.
Note that an up (respectively down) $k$-slide on a plane tree gives a smaller (respectively larger) plane tree in the $k$-associative order, since it corresponds to a left (respectively right) $k$-rotation of the corresponding binary trees.

\subsection{Results on modular Catalan numbers via trees}
We say a plane tree or binary tree is \demph{$k$-minimal} or \demph{$k$-maximal} if it is minimal or maximal in its $k$-equivalence class.
We investigate $k$-minimality and $k$-maximality for both types of trees.

\begin{proposition}\label{prop:plane1}
A plane tree is $k$-maximal (respectively $k$-minimal) if and only if every node (respectively non-root node) has degree $\le k$ (respectively $<k$).
Furthermore, each $k$-equivalence class of plane trees has a unique minimal representative.
\end{proposition}

\begin{proof}
The definition of up and down $k$-slides immediately implies the first statement.
By \eqref{eq:degree}, two plane trees $k$-equivalent to each other have congruent multi-degrees modulo $k$. 
This, together with the first statement, implies the second statement.
\end{proof}

Let $\defcolor{\comb_0}$ be the unique tree with a single node.
We recursively define the \demph{left comb of length $k$} to be $\defcolor{\comb_k}:=\comb_{k-1}\wedge \comb_0$ for $k\ge1$.
The binary tree $\defcolor{\comb_k^{1}} := \comb_0\wedge \comb_k$ is useful for calculating $C_{k,n}$.
The figures below give $\comb_4$ and $\comb_4^{1}$.
\begin{figure}[h]
\hskip-90pt\Tree[. [. [. [. {\nola} {\nola} ] {\nola} ] {\nola} ] {\nola} ] \quad
\Tree[. {\nola} !\qsetw{-.5in} [. [. [. [. {\nola} {\nola} ] {\nola} ] {\nola} ] {\nola} ] ]
\caption{$\comb_4$ and $\comb_4^1$}
\end{figure}

Proposition~\ref{prop:plane1} and Proposition~\ref{prop:d-del-corresp} give the following result, Proposition~\ref{prop:binary1}, via the bijection between plane trees and binary trees.

\begin{proposition}\label{prop:binary1}
A binary tree is $k$-maximal (respectively $k$-minimal) if and only if it avoids $\comb_{k+1}$ (respectively $\comb_k^1$) as a subtree.
Furthermore, each $k$-equivalence class of $\T_n$ has a unique minimal representative.
\end{proposition}

\begin{proposition}\label{prop:plane2}
Two plane trees are $k$-equivalent if and only if their multi-degrees are congruent modulo $k$.
\end{proposition}

\begin{proof}
The ``only if'' part follows from \eqref{eq:degree}.
For the ``if'' part, suppose $T_1$ and $T_2$ are two plane trees with multi-degrees $\md(T_1) \equiv \md(T_2) \mod k$.
For $i=1,2$ let $T'_i$ be the unique minimal tree $k$-equivalent to $T_i$.
Then $\md(T'_1)\equiv \md(T'_2)\mod k$, which implies $\md(T'_1)=\md(T'_2)$ by Proposition~\ref{prop:plane1}.
Hence $T_1$ and $T_2$ are both $k$-equivalent to the same minimal representative.
\end{proof}

Proposition~\ref{prop:plane2} implies Proposition~\ref{prop:binary2}.

\begin{proposition}\label{prop:binary2}
Two binary trees are $k$-equivalent if and only if their left depths are congruent modulo $k$.
\end{proposition}

Using Proposition~\ref{prop:binary2}, we compute $\{C_{2,n}\}$.

\begin{proposition}
We have $C_{2,n}=2^{n-1}$ for $n\ge1$.
\end{proposition}

\begin{proof}
Since any list in $\mathcal D_n$ is of the form $(\ld_0,\ldots,\ld_{n-2},1,0)$, and $\ld_i\equiv 0$ or $1$ mod $2$ for each $i$, the number of equivalence classes modulo $2$ in $\mathcal D_{n}$ is at most $2^{n-1}$.
We prove this upper bound is sharp by induction on $n$.
For this we assume $C_{2,n-1}=2^{n-2}$.
Let $t\in\T_{n-1}$ with left depth $\ld(t) = (d_0,\ldots,d_{n-3},1,0)$.
Giving two children to the leaf labeled $n-1$ or to the leaf labeled $n-2$ of $t\in\T_{n-1}$ gives two different trees in $\T_n$ whose left depths are $(d_0,\ldots,d_{n-3},1,1,0)$ and $(d_0,\ldots,d_{n-3},2,1,0)$.
Hence $C_{2,n}\ge 2 C_{2,n-1} = 2^{n-1}$.
\end{proof}

Define $\defcolor{\T_{k,n}}$ to be the maximal subset of $\T_n$ whose members avoid $\comb_k^1$.
That is, $\T_{k,n}$ is the set of all $k$-minimal binary trees with $n$ internal nodes.
The next result is a consequence of Proposition~\ref{P:number_classes_is_Catalan} and Proposition~\ref{prop:binary1}.
\begin{corollary}\label{cor:Cnk}
The modular Catalan number $\C_{k,n}$ enumerates $\T_{k,n}$.
\end{corollary}

Equations~\eqref{eq:C1} and \eqref{eq:C2} in Section~\ref{sec:formula} give closed formulas  for $C_{k,n}$.
Directly counting trees in $\T_n$ containing $\comb_k^1$ gives special cases of Equation~\eqref{eq:C2}:
we have $C_{k,n} = C_n$ for $n\le k$ and 
\[ C_{k, k+\ell} = C_{k+\ell}-\tbinom{k+2\ell}{\ell-1} \qquad \text{if}\quad k\ge\ell\ge1. \]

For $k\ge0$ we define the \demph{generalized Motzkin number} $\defcolor{M_{k,n}}$ to be the number of binary trees in $\T_n$ avoiding $\comb_{k+1}$.
When $k\ge1$ the number $M_{k,n}$ enumerates $k$-maximal elements of $\T_n$.
Proposition~\ref{prop:binary1} implies $M_{k-1,n}\le C_{k,n}\le M_{k,n}$.
One sees that $M_{0,0}=1$, $M_{0,n}=0$ for $n\ge1$, and $M_{1,n}=1$ for $n\ge0$. 
More generally, Equations~\eqref{eq:M1} and \eqref{eq:M2} of Section~\ref{sec:formula} are closed formulas for $M_{k,n}$.
These formulas could also be derived from work of Tak\'acs~\cite{Takacs} on plane trees with degree constraints.
Directly counting trees in $\T_n$ containing $\comb_{k+1}$ gives specializations of Equation~\eqref{eq:M2}: we have $M_{k,n}=C_n$ for $n\le k$ and 
\[ M_{k, k+\ell} = C_{k+\ell}-\tbinom{k+2\ell-1}{\ell-1} \qquad \text{if}\quad k\ge\ell-1\ge0.\]

The following is a corollary to Proposition~\ref{prop:plane2}.

\begin{corollary}\label{cor:plane}
The modular Catalan number $C_{k,n}$ enumerates plane trees with $n+1$ nodes whose non-root nodes have degree less than $k$.
The generalized Motzkin number $M_{k,n}$ enumerates plane trees with $n+1$ nodes, each having degree no more than $k$.
\end{corollary}

We denote by $\defcolor{[T]_k}$ the $k$-equivalence class of a plane tree $T$.
We next study the largest $k$-equivalence classes.
Let $T$ be a plane tree with multi-degree $\md(T) = (\md_0,\ldots,\md_n)$.
Assume $\md_j\ge1$ for some $j\in[n]$.
By Proposition~\ref{prop:degree}, subtracting $1$ from $\md_j$ and adding $1$ back to $\md_0$ still gives a multi-degree of some plane tree, which is denoted by $\defcolor{\phi_j(T)}$.

\begin{lemma}\label{lem:injection}
Suppose $T$ is a $k$-minimal plane tree with $\md(T)=(\md_0,\ldots,\md_n)$.
Assume $\md_j\ne 0$ for some $j\in[n]$. 
Then we have the following.

\noindent (i) 
The tree $\phi_j(T)$ is also $k$-minimal.

\noindent (ii)
Sending $T'$ to $\phi_j(T')$ for all $T'\in[T]_k$ gives an injection $\phi_j: [T]_k \hookrightarrow [\phi_j(T)]_k$.

\noindent(iii)
The above injection $\phi_j$ is a bijection if and only if the multi-degree $(a_0,\ldots,a_n)$ of every tree in $[\phi_j(T)]_k$ satisfies $a_0+\cdots+a_i\ge i+1$ for all $i\in\{0,1,\ldots,j-1\}$.
\end{lemma}

\begin{proof}
Let the multi-degree of $\phi_j(T)$ be $(e_0,\ldots,e_n)$. 
Then $e_i\le \md_i$ for all $i\in[n]$.
This implies that $\phi_j(T)$ is $k$-minimal by Proposition~\ref{prop:plane2}.
If $T'$ is $k$-equivalent to $T$ then its multi-degree $\md(T') = (\md'_0,\ldots,\md'_n)$ satisfies
$\md'_j\ge1$ since $\md'_j\equiv \md_j\mod k$.
Hence $\phi_j(T')$ is well defined and has multi-degree $(e'_0,\ldots,e'_n)$ congruent to $\md(\phi_j(T))=(e_0,\ldots,e_n)$ modulo $k$.
Then we have a well defined map $\phi_j: [T]_k \to [\phi_j(T)]_k$, which is an injection since subtracting one from $e'_0$ and adding one back to $e'_j$ gives the unique preimage of $\phi_j(T')$. 
Combining this with Proposition~\ref{prop:degree} also shows that there exists an inverse of the injection $\phi_j$ if and only if the multi-degree $(a_0,\ldots,a_n)$ of each tree in $[\phi_j(T)]_k$  satisfies $a_0+\cdots+a_i\ge i+1$ for all $i\in\{0,1,\ldots,j-1\}$.
\end{proof}

Using Lemma~\ref{lem:injection} we can find all the largest $k$-equivalence classes in $\T_{n}$.
Let $m$ be the smallest positive integer congruent to $n$ modulo $k$.
A plane tree $T$ with multi-degree $\md(T)=(\md_0,\ldots,\md_n)$ is called \demph{$k$-admissible} if 
\begin{itemize}
\item
$(\md_0-n+m, \md_1,\ldots,\md_m)$ is the multi-degree of some plane tree with $m+1$ nodes and $\md_{m+1}=\cdots=\md_n=0$, or equivalently,
\item
$(\md_0-n+m,\md_1,\ldots,\md_{m-1},\md_m+n-m,0,\ldots,0)$ is the multi-degree of some tree in $[T]_k$.
\end{itemize}
If $T$ is $k$-admissible then $T$ is $k$-minimal since $\md_1,\ldots,\md_m<m\le k$.
For example, the unique plane tree $\defcolor{T(n,0,\ldots,0)}$ with multi-degree $(n,0,\ldots,0)$ is $k$-admissible, and for $k=3$ and $n=6$ the $k$-admissible plane trees with $n+1$ nodes have the following multi-degrees:
\[ (6,0,0,0,0,0,0),\
(5,1,0,0,0,0,0),\
(5,0,1,0,0,0,0),\
(4,2,0,0,0,0,0),\
(4,1,1,0,0,0,0).\]

\begin{theorem}\label{thm:largest}
Fix $n\ge0$ and $k\ge1$. Let $m$ be the smallest positive integer congruent to $n$ modulo $k$.
Then a $k$-equivalence class of plane trees with $n+1$ nodes has the largest size if and only if its minimal representative is $k$-admissible.
\end{theorem}

\begin{proof}
Lemma~\ref{lem:injection} gives a chain of injections from any $k$-equivalence class of plane trees with $n+1$ nodes to $[T(n,0,\ldots,0)]_k$.
Hence $[T(n,0,\ldots,0)]_k$ has the largest size among all $k$-equivalence classes of plane trees with $n+1$ nodes.

Let $T$ be a $k$-minimal plane tree with multi-degree $\md(T)=(\md_0,\ldots,\md_n)\ne(n,0,\ldots,0)$.
Then $\md_j\ge1$ for some $j\in[n]$.
Let $(e_0,\ldots,e_n)$ be the multi-degree of $\phi_j(T)$.
Suppose the injection $\phi_j:[T]_k\hookrightarrow [\phi_j(T)]_k$ is a bijection and $\phi_j(T)$ is $k$-admissible.
We have 
\[ (e_0-n+m,e_1,\ldots,e_{m-1},e_m+n-m,0,\ldots,0) \equiv (e_0,\ldots,e_n) \mod k\]
where the left hand side is the multi-degree of some tree in $[\phi_j(T)]_k$ by definition of the $k$-admissibility.
This tree has a preimage under the bijection $\phi_j$, and the multi-degree of the preimage must be $(\md_0-n+m,\md_1,\ldots,\md_{m-1},\md_m+n-m,0,\ldots,0)$.
This implies $T$ is $k$-admissible.
Hence if $[T]_k$ has the same size as $[T(n,0,\ldots,0)]_k$ then $T$ must be $k$-admissible.

Now suppose $T$ is indeed $k$-admissible.
Then $\md_j\ge1$ implies $j\le m$. 
Thus $\phi_j(T)$ also $k$-admissible by definition. 
Let $(e'_0,\ldots,e'_n)$ be the multi-degree of any tree in $[\phi_j(T)]_k$. 
For any $i<j$ we have 
\[ e'_0+\cdots+e'_i \equiv \md_0-n+m+\md_1+\cdots+\md_i+1 \mod k.\]
Since $T$ is $k$-admissible, we also have $\md_0-n+m+\md_1+\cdots+\md_i\ge i$.
Combining these with $i<j\le m\le k$ we obtain $e'_0+\cdots+e'_i \ge i+1$. 
Hence $\phi_j:[T]_k\to[\phi_j(T)]_k$ is a bijection by Lemma~\ref{lem:injection}.
This implies that any $k$-admissible plane tree represents a $k$-equivalence class of equal size as $T(n,0,\ldots,0)$.
\end{proof}

\begin{corollary}\label{cor:largest}
Fix $n\ge0$ and $k\ge1$. 
Let $m$ be the smallest positive integer congruent to $n$ modulo $k$.
Among all $k$-equivalence classes of plane trees with $n+1$ nodes, there are $C_m$ many that have the largest size, one of which is represented by $T(n,0,\ldots,0)$.
\end{corollary}

Finally, the size of the largest $k$-equivalence classes of plane trees with $n+1$ nodes will be given in Proposition~\ref{prop:FormulaLargest}.

\section{Connections with other objects}\label{sec:objects}

We explore $M_{k,n}$ and $C_{k,n}$ as they pertain to other Catalan objects.
A \demph{Dyck path of (semi)length $2n$} is a diagonal lattice path from $(0,0)$ to $(2n,0)$ consisting of $n$ up-steps $U=(1,1)$ and $n$ down-steps $D=(1,-1)$ such that none of the path is below the $x$-axis.
Every sequence $\md=(\md_0,\ldots,\md_n)$ of nonnegative integers corresponds to a lattice path
\[ L(\md):=U^{\md_0}DU^{\md_1}\cdots DU^{\md_n} \]
which is a Dyck path if and only if $\md$ is the multi-degree of a plane tree.
This gives a bijection between plane trees with $n+1$ nodes and Dyck paths of length $2n$.

A \demph{partition} is a decreasing sequence of nonnegative integers $\lambda=(\lambda_1,\ldots,\lambda_n)$.
The \demph{size} of $\lambda$ is $|\lambda|:=\lambda_1+\cdots+\lambda_n$ and the \demph{length} of $\lambda$ is $\ell(\lambda):=\#\{i\in[n]:\lambda_i>0\}$.
It is often convenient to represent $\lambda$ by its \demph{Young diagram}, which has $\ell(\lambda)$ many left-justified rows with $\lambda_i$ boxes on the $i$th row for $i=1,2,\ldots,\ell(\lambda)$.
See Figure~\ref{fig:tableau} below.
Say a partition $\lambda$ is \demph{bounded by} another partition $\mu$ and write $\defcolor{\lambda\subseteq \mu}$ if the Young diagram of $\lambda$ is contained in the Young diagram of $\mu$.
The partition $\defcolor{k^n}:=(k,\ldots,k)$ is a sequence of $n$ copies of $k$.

A Dyck path of length $2n$ may also be written as $L=UD^{e_1}UD^{e_2}\cdots UD^{e_n}$. 
It corresponds to a partition $\lambda(L):=(\lambda_1(L),\ldots,\lambda_n(L))$ whose $j$th part $\lambda_j(L):=e_1+\cdots+e_{n-j}$ satisfies $0\le \lambda_j\le n-j$ for all $j\in[n]$.
Thus $L\mapsto \lambda(L)$ gives a bijection between Dyck paths of length $2n$ and partitions of the form $\lambda=(\lambda_1,\ldots,\lambda_n)$ with $0\le\lambda_j\le n-j$ for all $j\in[n]$.
The Young diagram of $\lambda(L)$ is enclosed between Dyck paths $L$ and $U^nD^n$.
Thus $\lambda(L)$ is bounded by $(n-1,n-2,\ldots,1,0)$.

There is also a simple bijection between Dyck paths of length $2n$ and $2\times n$ standard Young tableaux.
For each $i\in[2n]$, if the $i$th step is up (respectively down) in the Dyck path then put $i$ on the top (respectively bottom) row of the corresponding tableau.
See, e.g.,~\cite{Yong} for more information on Young diagrams and Young tableaux.

An example of the correspondence among plane trees with $n+1$ nodes, Dyck paths of length $2n$, partitions with $n$ nonnegative parts bounded by $(n-1,n-2,\ldots,1,0)$, and $2\times n$ standard Young tableaux is given below ($n=4$).
\begin{figure}[h]
\raisebox{12pt}{$\begin{array}{c}\  \raisebox{8pt}{$T=$}\ \includegraphics[scale=.32]{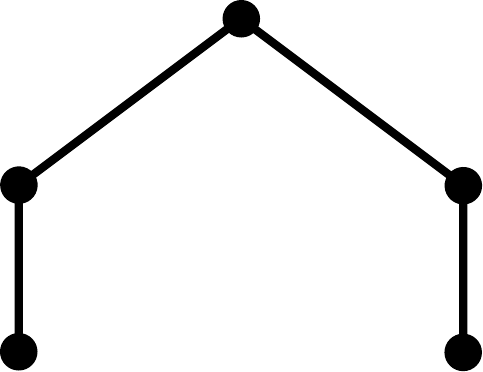} \\ \md(T) = (2,1,0,1,0) \end{array}$ }\qquad
\raisebox{40pt}
 {$ \xymatrix @R=8pt @C=8pt @M=-1pt
  { & & & & \nola \ar@{.}[rrrddd] \ar@{.}[lldd]  \\
   & & & \txt{$\phantom{\bullet}$} \ar@{.}[rd] & & \txt{$\phantom{\bullet}$} \ar@{.}[ld]  \\
   & & \bullet \ar@{-}[ld] \ar@{-}[rd] & & \bullet \ar@{-}[rd]  & & \txt{$\phantom{\bullet}$} \ar@{.}[ld]\\
   & \bullet \ar@{-}[ld] & & \bullet \ar@{-}[ru]  & & \bullet \ar@{-}[rd]   & & \bullet \ar@{-}[rd]   \\
   \bullet & & & & & & \bullet \ar@{-}[ru]  & & \bullet } 
 $}
 \qquad \raisebox{10pt}{$\lambda=(3,1,0,0)$}
 \quad
 \young(1247,3568)
\caption{Correspondence among plane trees, Dyck paths, partitions, and tableaux}\label{fig:tableau}
\end{figure}

Now we discuss pattern avoidance for permutations.
Denote by ${\defcolor{\S_n}}$ the symmetric group consisting of all permutations of $[n]$, and write a permutation $w\in\S_n$ as a word $w(1)\cdots w(n)$.
Let $w\in\S_n$ and $u\in\S_m$ with $m\le n$.
Say $w$ \demph{contains the pattern $u(1)$-$u(2)$-$\cdots$-$u(m)$} if there exists $1\le r_1<\cdots<r_m\le m$ such that $w(r_i)<w(r_j) \Leftrightarrow u(i)<u(j)$ whenever $1\le i<j\le m$.
Moreover, if we omit a dash between $u(j)$ and $u(j+1)$ in the above definition then $w(r_j)$ and $w(r_{j+1})$ are required to be adjacent entries of $w$, i.e., $r_{j+1}=r_j+1$.
Say $w$ \demph{avoids} a pattern if it does not contain that pattern.

Given a word $w=w_1\cdots w_n$ of distinct numbers $w_1,\ldots,w_n$, we construct a binary tree $\tr(w)\in\T_n$ whose internal nodes are labeled by $w_1,\ldots,w_n$.
Suppose $w_i$ is maximal among $w_1,\ldots,w_n$.
We draw the root of $\tr(w)$, labeling it $w_i$, and we recursively construct two binary trees $\tr(w_1\cdots w_{i-1})$ and $\tr(w_{i+1}\cdots w_n)$ and label their internal nodes. 
We then attach these trees to the root of $\tr(w)$ as left and right subtrees.
Restricting the map $\tr$ to $\S_n$ gives a poset surjection onto the Tamari lattice $\T_n$, where $\S_n$ is partially ordered by the \demph{weak order}\,: $u<v$ if the inversions of $u$ are contained in the inversions of $v$..

Conversely, for each $t\in \T_n$,  one obtains a permutation, denoted by $\defcolor{\tr^{-1}(t)}$, by labeling the internal nodes with $n,n-1,\ldots,1$ according to the pre-order and then reading these labels following the in-order.
Here the \demph{in-order} recursively lists first the left subtree of the root, next the root itself, and last the right subtree of the root.
One can check that $\tr(\tr^{-1}(t))=t$ for any $t\in\T_n$.
Hence $\tr:\S_n\twoheadrightarrow \T_n$ is a poset surjection and $\tr^{-1}: \T_n\hookrightarrow \S_n$ is a poset injection.
Moreover, the image of $\tr^{-1}$ is the set of ($1$-$3$-$2$)-avoiding permutations in $\S_n$ (cf. Exercise~\cite[6.19.ff]{EC2}).

\begin{example}\label{example:tr}
The left hand tree in Figure~\ref{fig:tr} is $t=\tr(26513874)$.
One sees that $\tr^{-1}(t)$ is the ($1$-$3$-$2$)-avoiding permutation $67534821$ and $\tr(67534821)=t$ by the right hand picture in Figure~\ref{fig:tr}.
\begin{figure}[h]
\[ \begin{picture}(140,110) 
  \put(10,0){\includegraphics[scale=.32]{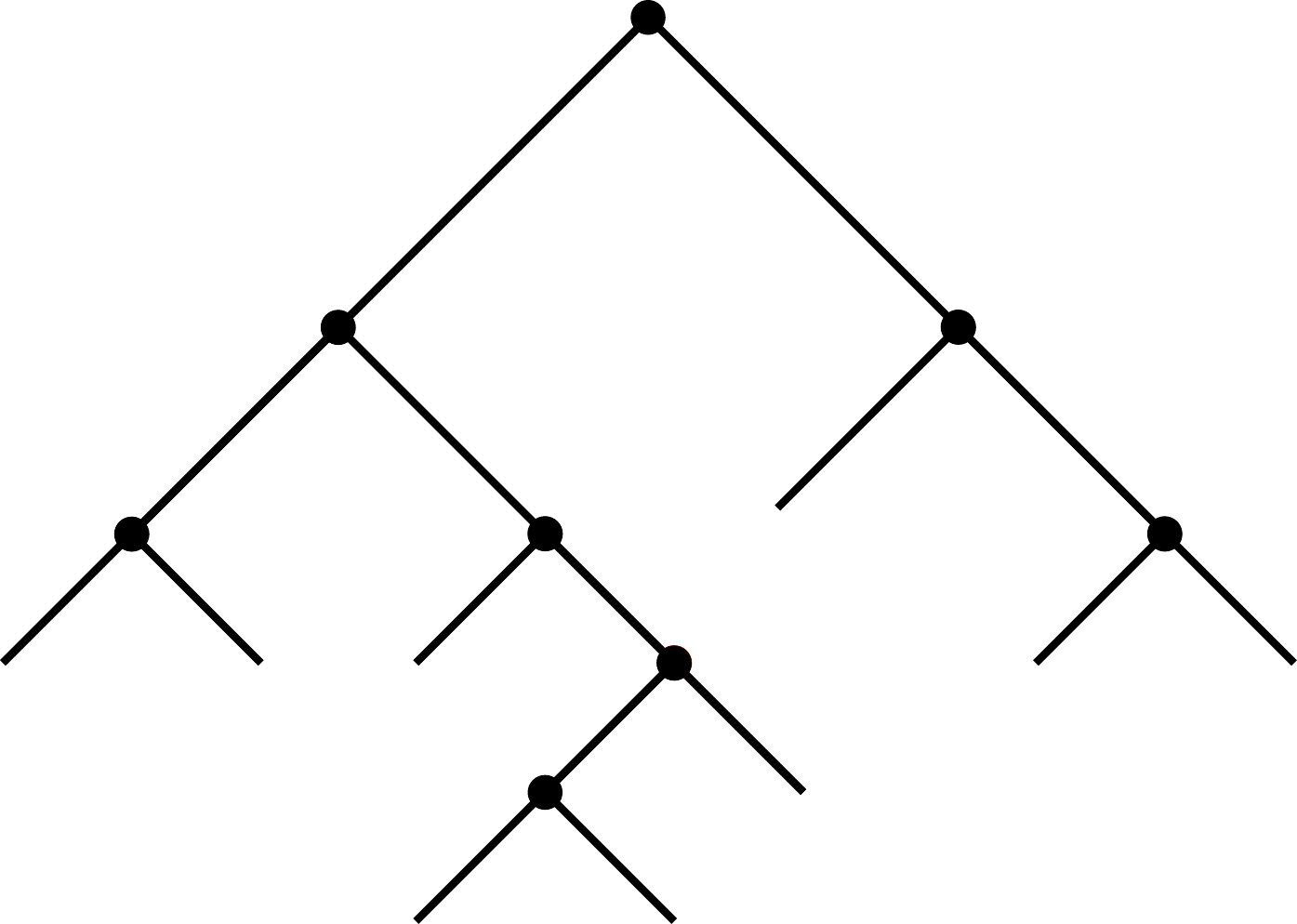}}
  \put(57,15){1}
  \put(15,40){2}
  \put(79,28){3}
  \put(37,63){6}
  \put(66.5,40){5}
  \put(129,40){4}
  \put(108,63){7}
  \put(77,93){8}
 \end{picture}
 \rule{60pt}{0pt}
 \begin{picture}(140,110) 
  \put(10,0){\includegraphics[scale=.32]{binaryLabeled.pdf}}
  \put(57,15){3}
  \put(15,40){6}
  \put(79,28){4}
  \put(37,63){7}
  \put(66.5,40){5}
  \put(129,40){1}
  \put(108,63){2}
  \put(77,93){8}
 \end{picture} \]
\caption{The maps $\tr$ and $\tr^{-1}$}\label{fig:tr}
\end{figure}
\end{example}

The bijections described earlier lead to connections between the relevant Catalan objects and the numbers $M_{k,n}$ and $C_{k,n}$.

\begin{proposition}\label{prop:M counts}
For $n\ge0$ and $k\ge1$, $M_{k-1,n}$ enumerates the following:
\begin{enumerate}
\item binary trees with $n$ internal nodes avoiding $\comb_k$,
\item plane trees with $n+1$ nodes, each having degree less than $k$,
\item Dyck paths of length $2n$ avoiding $U^{k}$ ($k$ consecutive up-steps).
\item partitions bounded by $(n-1,n-2,\ldots,1,0)$ with each part occurring fewer than $k$ times,
\item $2\times n$ standard Young tableaux avoiding $k$ consecutive numbers in the top row, and
\item permutations of $[n]$ avoiding $1$-$3$-$2$ and $12\cdots k$.
\end{enumerate}
\end{proposition}

\begin{proof}
That $M_{k-1,n}$ enumerates the sets (1)--(5) follows directly from its definition and the descriptions of the appropriate bijections.
To reveal the more obscure result, that $M_{k-1,n}$ counts (6), we apply the involution $w\mapsto w^{-1}$ on permutations avoiding $1$-$3$-$2$.
Let $t$ be a binary tree corresponding to a ($1$-$3$-$2$)-avoiding permutation $w\in\S_n$.
If $t$ contains $\comb_{k}$ then there exist $1\le i_1<\cdots< i_k\le n$ with $w(i_1)<\dotsb<w(i_k)$ consecutive increasing integers, say $w(i_j)=h+j$ for some $h\in [n-k]$ and all $j\in[k]$.
Equivalently, if $t$ contains $\comb_k$ then there exists $h\in [n-k]$ such that
\[ w^{-1}(h+1)<w^{-1}(h+2)<\cdots<w^{-1}(h+k).\]
Hence $t$ avoids $\comb_k$ if and only if $w^{-1}$ avoids $1$-$3$-$2$ and $12\cdots k$.
\end{proof}

Similarly one can prove the following result.

\begin{proposition}\label{prop:C counts}
For $n\ge0$ and $k\ge1$, $C_{k,n}$ enumerates the following:
\begin{enumerate}
 \item the set $\T_{k,n}$ of binary trees with $n$ internal nodes avoiding $\comb_k^1$,
 \item plane trees with $n+1$ nodes whose non-root nodes have degree less than $k$,
 \item Dyck paths of length $2n$ avoiding $DU^k$ (a down-step immediately followed by $k$ up-steps),
 \item partitions bounded by $(n-1,n-2,\ldots,1,0)$ with each positive part occurring fewer than $k$ times,
 \item $2\times n$ standard Young tableaux which contain no list of $k$ consecutive numbers in the top row other than $1,2,\ldots,\ell$ for any $\ell\in[n]$,
 \item permutations of $[n]$ avoiding $1$-$3$-$2$ and $23\cdots (k+1)1$.
\end{enumerate}
\end{proposition}

Next we describe a well-known surjection from the Tamari lattice $\T_n$ to the \demph{Boolean lattice} $\defcolor{\mathcal B_{n-1}}$ consisting of subsets of $[n-1]$ ordered by containment. 
Given $t\in \T_n$, define $\des(t)$ to be the set of all $i\in[n-1]$ such that the $(i+1)$th leaf of $t$ is a right child.
In other words, if $(\md_0,\ldots,\md_n)$ is the multi-degree of the plane tree corresponding to $t$ then $\des(t):=\{i\in[n-1]: \md_i>0\}$.
This gives a poset surjection $\des: \T_n\twoheadrightarrow \mathcal B_{n-1}$. 
Moreover,  for any permutation $w\in \S_n$ one can check that $\des(\tr(w))$ equals the \demph{descent set} $\{i\in[n-1]:w(i)>w(i+1)\}$ of  $w$.
For more details see, for example, Loday and Ronco~\cite{LodayRonco}.

Now we define a map $\des^{-1}: \mathcal B_{n-1}\to \T_n$ as follows. 
Let $S=\{i_1,\ldots,i_h\} \in\mathcal B_{n-1}$, where $i_1<\cdots<i_h$. 
Then $\des^{-1}(S)$ is the binary tree whose corresponding plane tree has multi-degree $(\md_0,\ldots,\md_n)$ satisfying $\md_0=n-|S|$, $\md_i=1$ if $i\in S$, and $\md_j=0$ if $j\in[n]\setminus S$.

\begin{proposition}\label{prop:des}
The map $\des^{-1}:\mathcal B_{n-1}\to\T_n$ is an order-preserving injection.
In particular, $\mathcal B_{n-1}\cong\des^{-1}(\mathcal B_{n-1})$ is a lattice isomorphism.
Furthermore, $\des^{-1}(\mathcal B_{n-1})=\T_{2,n}$, and for each $S\in \mathcal B_{n-1}$, $\des^{-1}(S)$ is the unique minimal element of the fiber $\{ t\in\T_n: \des(t)=S\}$ under the Tamari order.
\end{proposition}

\begin{proof}
Suppose $R\subsetneq S$ is a covering relation in the Boolean lattice $\mathcal B_{n-1}$ with $S\setminus R=\{i\}$.
Let $T$ and $T'$ be plane trees corresponding $\des^{-1}(R)$ and $\des^{-1}(S)$.
Then their multi-degrees satisfies $\md_0(T) = \md_0(T')+1$, $\md_i(T) = 0$, $\md_i(T')=1$, and $\md_j(T)=\md_j(T')$ for all $j\in[n]\setminus\{i\}$.
We may obtain $T$ from $T'$ by a series of up $1$-slides.
Hence $\des^{-1}(R)< \des^{-1}(S)$ in the Tamari order. This shows that $\des^{-1}$ is order-preserving.

Next, let $S\in\mathcal B_{n-1}$.
One can check that $\des(\des^{-1}(S))=S$.
Hence $\des^{-1}$ is injective and we have the isomorphism $\mathcal B_{n-1}\cong\des^{-1}(\mathcal B_{n-1})$ of lattices.

The multi-degree of the plane tree corresponding to $\des^{-1}(S)$ is $2$-minimal by Proposition~\ref{prop:plane1}.
Hence $\des^{-1}(S)\in\T_{2,n}$.
We have $|\T_{k,n}|=C_{2,n}=2^{n-1}=|\mathcal B_{n-1}|$.
Thus $\des^{-1}(\mathcal B_{n-1})=\T_{2,n}$.

Finally, let $t\in\T_n$ whose corresponding plane tree is $T$.
If a non-root node of $T$ has degree at least $2$ then applying an up $1$-slide at this node gives another plane tree $T'$ whose corresponding binary tree $t'$ satisfies $\des(t')=\des(t)$.
Thus any minimal element of the fiber $\{t\in \T_n: \des(t)=S\}$ avoids $\comb_2^1$ and must be the tree $\des^{-1}(S)$.
\end{proof}

\section{Closed formulas}\label{sec:formula}

We derive closed formulas for $C_{k,n}$ and $M_{k,n}$ from generating functions.
Let $\mathcal{PT}_n$ denote the set of plane trees with $n{+}1$ nodes.  Given $T\in\mathcal{PT}_n$ with multi-degree $\md(T)=(\md_0,\ldots,\md_n)$, we define $\bx_T := x_{d_0}\cdots x_{d_n}$.  We define a generating function,
\[ C(\bx,z) := \sum_{n\ge0} \sum_{\mathcal{PT}_n} \bx_T z^{n+1}\,.\]
To study this generating function we need the following Lagrange inversion formula.

\begin{theorem}[{Stanley~\cite[Theorem~5.4.2]{EC2}}]
Suppose that $A(z)$ and $B(z)$ are formal power series in $z$ such that $A(0)=B(0)=0$ and $A(B(z))=z$. If $n$ and $\ell$ are integers then 
\[ n[z^n]B(z)^\ell = \ell[z^{n-\ell}] (z/A(z))^n.\]
\end{theorem}

\begin{proposition}\label{prop:C^l}
For $n\ge1$ and $\ell\ge0$ we have
\begin{eqnarray}
[z^{n}] C(\bx,z)^\ell &=& \frac\ell{n} \left[z^{n-\ell}\right] \left( x_0+x_1z+x_2z^2+\cdots \right)^{n} \label{eq:C^l1} \\
&=& \frac{\ell}{n}\sum_{\substack{ m_0+m_1+m_2+\cdots = n\\ m_1+2m_2+\cdots = n-\ell} } {n\choose m_0,m_1,m_2,\ldots } x_0^{m_0} x_1^{m_1} x_2^{m_2} \cdots. \label{eq:C^l2}
\end{eqnarray}
\end{proposition}

\begin{proof}
If $T$ is a plane tree whose root has degree $\ell$ then the multi-degree of $T$ contains $\ell$, followed by the multi-degrees of the $\ell$ subtrees of the root.
Hence
\begin{equation}\label{eq:C(x,z)}
C(\bx,z) = \sum_{\ell\ge0} z x_\ell C(\bx,z)^\ell.
\end{equation}
Applying Lagrange inversion to $A(z) := z/(x_0+x_1z+x_2z^2+\cdots)$ and $B(z) := C(\bx,z)$ gives the result.
\end{proof}

\begin{corollary}[{\cite[Theorem 5.3.10]{EC2}}]\label{cor:Krew}
Given nonnegative integers $\ell, m_0,m_1,m_2,\ldots$ with $m_0+m_1+\cdots =n\ge1$, the number of plane trees with a root of degree $\ell$ and $m_i$ non-root nodes of degree $i$ for $i=0,1,2,\ldots$ is 
\[ [z^nx_0^{m_0}x_1^{m_1}\cdots] C(\bx,z)^\ell = 
\begin{cases}
\displaystyle\frac{\ell}{n}{n\choose m_0,m_1,m_2,\ldots }, & \text{if } m_1+2m_2+\cdots=n-\ell, \\
0, & \text{otherwise}.
\end{cases}\] 
\end{corollary}

\begin{remark}
We use the Lagrange inversion formula to prove Proposition~\ref{prop:C^l}, which immediately implies Corollary~\ref{cor:Krew}.
Stanley directly proved Corollary~\ref{cor:Krew} \cite[Theorem 5.3.10]{EC2} and used it as one way to prove the Lagrange inversion formula \cite[the second proof of Theorem~5.4.2]{EC2}. 

Taking $\ell=1$ in Corollary~\ref{cor:Krew} recovers a well-known result:
if $m_0+m_1+\cdots = n\ge1$ and $m_1+2m_2+\cdots=n-1$ then plane trees with $m_i$ nodes of degree $i$'s, or equivalently, Dyck paths with $m_i$ occurrences $U^iD$, for all $i=0,1,2,\ldots$, are enumerated by the \demph{Kreweras number}~\cite{Athanasiadis, Kreweras, ReinerSommers}
\[ \mathrm{Krew}(0^{m_0}1^{m_1}2^{m_2}\cdots) := \frac1{n}{n\choose m_0,m_1,m_2,\ldots}. \]
\end{remark}

\begin{proposition}\label{prop:C_I}
Let $\ell,n\ge 1$ and let $I$ be a set of nonnegative integers. 
Then the number of plane trees whose multi-degree $(\md_0,\ldots,\md_n)$ satisfies $\md_0=\ell$ and $\md_1,\ldots,\md_n\in I$ is
\[ \frac\ell{n} \left[z^{n-\ell}\right] \left( \sum_{i\in I} z^i \right)^{n} 
= \frac{\ell}{n}\sum_{\substack{ \sum_{i\in I}m_i = n \\ \sum_{i\in I} im_i = n-\ell} } {n\choose m_i: i\in I }.\]
\end{proposition}

\begin{proof}
Taking $x_i=1$ for all $i\in I$ and $x_j=0$ for all $j\notin I$ in \eqref{eq:C^l1} and \eqref{eq:C^l2} gives the result.
\end{proof}

Now we study $M_{k,n}$ and $C_{k,n}$, as well as their generating functions
\[ M_k(z):=\sum_{n\ge0} M_{k,n} z^{n+1} \quad\text{and}\quad
C_k(z):=\sum_{n\ge0} C_{k,n} z^{n+1}. \]
It follows from work of Rowland~\cite[Theorem 1]{Rowland} on binary trees that the generating functions $M_k(z)$ and
$C_k(z)$ are algebraic.
In the same work, Rowland used \textsf{Mathematica} to compute explicit polynomial equations satisfied by $M_k(z)$ and $C_k(z)$ for $k\le 4$.
We generalize this to all $k\ge1$ using the specialization $\defcolor{C(1^{k+1},z)}$ of $C(\bx,z)$ at $x_0=\cdots=x_k=1$ and $x_{k+1}=x_{k+2}=\cdots=0$.

\begin{proposition}\label{prop:gen}
For $k\ge0$ we have
\begin{equation}\label{eq:Mk}
M_k(z) = z + zM_k(z) + zM_k(z)^2 + \cdots + zM_k(z)^k.
\end{equation}
For $k\ge1$ we have 
\begin{equation}\label{eq:MC}
C_k(z) = z + zM_{k-1}(z) + zM_{k-1}(z)^2 + \cdots = z/(1-M_{k-1}(z)) \quad\text{and} 
\end{equation}
\begin{equation}\label{eq:Ck}
(C_k(z)-z)^k-C_k(z)^k+C_k(z)^{k-1}-zC_k(z)^{k-2}=0.
\end{equation}
\end{proposition}

\begin{proof}
Since $M_k(z) = C(1^{k+1},z)$, we deduce \eqref{eq:Mk} from \eqref{eq:C(x,z)}.
Considering the subtrees of the root of a plane tree we have \eqref{eq:MC}, which implies $M_{k-1}(z) = (C_k(z)-z)/(C_k(z))$.
Substituting this into \eqref{eq:Mk} gives \eqref{eq:Ck}.
\end{proof}

Let $\lambda=(\lambda_1,\ldots,\lambda_n)$ be a partition with $m_i$ parts equal to $i$ for $i=0,1,2,\ldots$.
Then 
\begin{itemize}
\item
$|\lambda|=n$ if and only if $m_1+2m_2+\cdots+km_k=n$, and 
\item
$\lambda\subseteq k^n$ if and only if $m_0+\cdots+m_k=n$ and $m_{k+1}=m_{k+2}=\cdots=0$.
\end{itemize}
The \demph{monomial symmetric function} $m_\lambda(x_1,\ldots,x_n)$ is the sum of $x_1^{e_1}\cdots x_n^{e_n}$ for all rearrangement $(e_1,\ldots,e_n)$ of $\lambda$. 
Taking $x_1=\cdots=x_n=1$ in $m_\lambda$ gives the multinomial coefficient 
\[ m_\lambda(1^n) = {n\choose m_0,m_1,m_2,\ldots}. \]

\begin{corollary}\label{cor:FormulaPositive}
For $k,n\ge0$, we have
\begin{equation}\label{eq:M1}
M_{k,n} 
= \frac{1}{n+1} \sum_{\substack{ \lambda\subseteq k^{n+1}\\ |\lambda|=n }} m_\lambda(1^{n+1})\,.
\end{equation}
For $k,n\ge1$, we have
\begin{equation}\label{eq:C1}
C_{k,n} 
= \sum_{\substack{\lambda\subseteq(k-1)^n \\ |\lambda|<n }} \frac{n-|\lambda|}{n} m_\lambda(1^n)\,.
\end{equation}
\end{corollary}

\begin{proof}
If the root of a plane tree has degree $\ell$, then deleting the root gives $\ell$ plane trees.
Hence, taking $I=\{0,1,\ldots,k\}$ and $\ell=1$ in Proposition~\ref{prop:C_I} gives a formula for $M_{k,n-1}$ which is equivalent to \eqref{eq:M1}.
Combining \eqref{eq:M1} with \eqref{eq:MC} we have the formula \eqref{eq:C1} for $C_{k,n}$.
\end{proof}

\begin{theorem}\label{thm:FormulaAlt}
For $k\ge1$ and $n\ge0$, we have
\begin{equation}\label{eq:M2}
M_{k-1,n} = \frac{1}{n+1} \sum_{0\le j\le n/k } (-1)^j {n+1\choose j} {2n-jk \choose n}\,. 
\end{equation}
For $k,n\ge1$, we have
\begin{equation}\label{eq:C2}
C_{k,n} = \sum_{0\le j\le (n-1)/k} \frac{(-1)^j}{n} {n\choose j} {2n-jk\choose n+1}\,. 
\end{equation}
\end{theorem}

\begin{proof}
By \eqref{eq:C^l1} we have 
\begin{equation}\label{eq:M^l}
[z^n] C(1^k,z)^\ell = \frac{\ell}n [z^{n-\ell}] \frac{(1-z^k)^n}{(1-z)^n}\
=\frac\ell n \sum_{0\le j\le (n-\ell)/k } (-1)^j {n\choose j} {2n-\ell-1-jk \choose n-1}.
\end{equation}
Taking $\ell=1$ gives a formula for $M_{k-1,n-1}$ which is equivalent to \eqref{eq:M2}. 
By \eqref{eq:MC} and \eqref{eq:M^l},
\begin{eqnarray*}
C_{k,n} &=& \sum_{1\le \ell\le n} \frac\ell n\sum_{0\le j\le (n-\ell)/k} (-1)^j {n\choose j} {2n-\ell-1-jk \choose n-1} \\
&=& \sum_{0\le j\le (n-1)/k} \frac{(-1)^j}{n} {n\choose j} \sum_{1\le \ell\le n-jk} \ell {2n-\ell-1-jk\choose n-1}  \\
&=& \sum_{0\le j\le (n-1)/k} \frac{(-1)^j}{n} {n\choose j} {2n-jk\choose n+1}.
\end{eqnarray*} 
The last step above follows from the formula (taking $a=2n-jk$, $b=n+1$, and $r=1$)
\begin{equation}\label{eq:second}
 {a\choose b} = \sum_{r\le \ell\le a-b+r} {\ell\choose r} {a-\ell-1\choose b-r-1}
\end{equation}
which can be proved by choosing a subset of $b$ elements from the set $[a]$ with the number $\ell+1$ being the $(r+1)$th smallest chosen element.
\end{proof}

\begin{remark}
It seems difficult in general to solve for $C_k(z)$ directly from Equation \eqref{eq:Ck}.
One can apply Lagrange inversion to it and obtain a closed formula of $C_{k,n}$ for $k\ge 3$ and $n\ge1$. 
However, the result is more complicated than our previous formulas \eqref{eq:C1} and \eqref{eq:C2}.
\end{remark}

Next, we derive from the proof of Theorem~\ref{thm:FormulaAlt} a closed formula for the total number $\defcolor{D_{k,n}}$ of intersection points between all Dyck paths of length $2n$ avoiding $DU^k$ and the $x$-axis.

\begin{proposition}\label{prop:intersection}
For $k\ge1$ and $n\ge1$ we have
\[ D_{k,n} = \sum_{0\le j\le (n-1)/k} \frac{(-1)^j \cdot 2}{n} {n\choose j} {2n-jk+1\choose n+2}.\]
\end{proposition}

\begin{proof}
Since the number of intersection points between a Dyck path and the $x$-axis is one plus the degree of the root of the corresponding plane tree, it follows from the proof of Theorem~\ref{thm:FormulaAlt} that 
\[ D_{k,n} = \sum_{0\le j\le (n-1)/k} \frac{(-1)^j}{n} {n\choose j} \sum_{1\le \ell\le n-jk} (\ell+1)\ell {2n-\ell-1-jk\choose n-1}.\]
Applying \eqref{eq:second} with $a=2n-jk+1$, $b=n+2$, and $r=2$ gives the result.
\end{proof}

\begin{remark}
(i) Similarly to \eqref{eq:MC}, the generating function $\defcolor{D_k(z)} := \sum_{n\ge0} D_{k,n} z^{n+1}$ satisfies
\[ D_k(z) = \sum_{\ell\ge0} (\ell+1) zM_{k-1}(z)^\ell = z / (1-M_{k-1}(z))^2. \]
Since $C(z):=\sum_{n\ge0} C_n z^{n+1}$ satisfies $C(z) = z + C(z)^2$, taking $k\to\infty$ in the above equation gives
\[ \lim_{k\to\infty} D_k(z) = z/(1-C(z))^2 = z/(z/C(z))^2 = C(z)^2/z = C(z)/z - 1 \]
which recovers a well-known fact that the total number of intersection points between Dyck paths of length $2n$ and the $x$-axis is the Catalan number $C_{n+1}$.

(ii)
For $k=2$ and $k=3$, Proposition~\ref{prop:intersection} gives a new interpretation for the sequences~\cite[A045623, A036908]{OEIS}.
In particular, for $k=2$ and $n\ge1$, the formula in Proposition~\ref{prop:intersection} can be simplified to $(n+3)2^{n-2}$, which enumerates various other interesting objects.
Also, since $M_1(z) = z/(1-z)$ and $M_2(z) = (1-z - \sqrt{1-2z-3z^2})/2z$, we have  
\[ D_2(z) = \frac{z(1-z)^2}{(1-2z)^2} \quad\text{and}\quad
D_3(z) = \frac{4z^3}{(3z-1+\sqrt{1-2z-3z^2})^2}.\]
We do not find any result related to Proposition~\ref{prop:intersection} for $k\ge4$ in the literature.
\end{remark}

Finally, we provide a formula for the largest size of $k$-equivalence classes.

\begin{proposition}\label{prop:FormulaLargest}
The largest size of a $k$-equivalence class of plane trees with $n+1$ nodes is 
\[ \sum_{0\le j\le n/k} \frac{n-jk}{n}{n+j-1 \choose j}. \]
\end{proposition}

\begin{proof}
By Theorem~\ref{thm:largest}, the plane tree with multi-degree $(n,0,\ldots,0)$ is the minimal element of a $k$-equivalence class of the largest size.
By Proposition~\ref{prop:plane2}, the plane trees with $n+1$ nodes belonging to this $k$-equivalence class are those whose multi-degree is congruent to $(n,0,\ldots,0)$ modulo $k$. 
Setting $I=\{0,k,2k,\ldots\}$ and $\ell\in\{n-jk: 0\le j\le n/k\}$ in Proposition~\ref{prop:C_I} demonstrates that such plane trees are enumerated by 
\[ \sum_{0\le j\le n/k} \frac{n-jk}{n}[z^{jk}](1-z^k)^{-n}. \]
Applying a binomial expansion gives the result.
\end{proof}

\section{Proofs by Dyck paths}\label{sec:Dyck}

We use Dyck paths to prove the closed formulas obtained in Section~\ref{sec:formula}.
Recall that every sequence $\mathbf{e}=(e_0,\ldots,e_n)$ of nonnegative integers corresponds to a lattice path 
\[ L(\mathbf{e}):=U^{e_0}DU^{e_1}\cdots DU^{e_n} \]
which is a Dyck path if and only if $\mathbf{e}$ is the multi-degree of a plane tree.
Assume the length of the lattice path $L=L(\mathbf{e})$ is $2n$.
For each $r\in\{0,1,\ldots,n\}$ we define a cyclic reordering of $L$,
\[  {\defcolor{L^{(r)}}} := U^{e_0} DU^{e_{r+1}}\cdots DU^{e_n}DU^{e_1}\cdots DU^{e_{r}}. \] 
We note that $L^{(0)} = L^{(n)}$.
Suppose the lowest point on the subpath $L' = DU^{e_1}DU^{e_2}\cdots DU^{e_n}$ has height $h$. 
For each $i\in[e_0]$, the line $y=h+i-1$ intersects $L'$, and the leftmost intersection point must be the end point of the $(r_i+1)$th down-step of $L$ for a unique integer $r_i=r_i(L)\in\{0,1,\ldots,n-1\}$.
It follows that $r_1(L)>\cdots>r_{e_0}(L)$.

\begin{lemma}\label{lem:Dyck1}
Let $L = U^{e_0}DU^{e_1}\cdots DU^{e_n}$ be a lattice path of length $2n$.
For $0\le r\le n-1$, $L^{(r)}$ is a Dyck path if and only if $r\in\{r_1(L),\ldots,r_{e_0}(L)\}$.
\end{lemma}
      
\begin{proof}
Decompose $L$ into subpaths $A=U^{e_0}$, $B= DU^{e_1}\cdots DU^{e_{r}}$ and $C= DU^{e_{r+1}}\cdots DU^{e_n}$.
Assume the initial point of $C$ is $(a,b)$, which is also the initial point of the $(r+1)$th down-step of $L$.
This down-step becomes the first down-step of $L^{(r)}$.
One can check that $L^{(r)}$ is a Dyck path if and only if $B$ is weakly above $y=b$ and $C$ is weakly above $y=b-e_0$.
This is also equivalent to saying that the $r=r_i(L)$ for some $i\in[e_0]$. 
The result follows.
\end{proof}

\begin{example}
Figure~\ref{fig:reorder} shows a lattice path $L$ and the lattice paths $L^{(r)}$ for $r=1,2,3$. 
While $L^{(1)}=L^{(3)}$ is a Dyck path, $L=L^{(2)}$ is not.
\begin{figure}[h]
\[\begin{array}{cccc}
\xymatrix @R=8pt @C=8pt @M=-1.5pt { 
& & \bullet \ar@{-}[ld] \ar@{-}[rd] \\
& \bullet \ar@{-}[ld] & & \bullet \ar@{-}[rd] & & \bullet \ar@{-}[ld] \ar@{-}[rd] \\
\bullet & & & & \bullet & & \bullet \ar@{-}[rd] & & \bullet \ar@{-}[ld] \\
& & & & & & & \bullet 
} & 
\xymatrix @R=8pt @C=8pt @M=-1.5pt { 
& & \bullet \ar@{-}[ld] \ar@{-}[rd] & & \bullet \ar@{-}[ld] \ar@{-}[rd] \\
& \bullet \ar@{-}[ld] & & \bullet & & \bullet \ar@{-}[rd] & & \bullet \ar@{-}[ld] \ar@{-}[rd] \\
\bullet & & & & & & \bullet & & \bullet
} & 
\xymatrix @R=8pt @C=8pt @M=-1.5pt { 
& & \bullet \ar@{-}[ld] \ar@{-}[rd] \\ 
& \bullet \ar@{-}[ld] & & \bullet \ar@{-}[rd] & & \bullet \ar@{-}[ld] \ar@{-}[rd] \\
\bullet & & & & \bullet & & \bullet \ar@{-}[rd] & & \bullet \ar@{-}[ld] \\
& & & & & & & \bullet } &
\xymatrix @R=8pt @C=8pt @M=-1.5pt { 
& & \bullet \ar@{-}[ld] \ar@{-}[rd] & & \bullet \ar@{-}[ld] \ar@{-}[rd] \\
& \bullet \ar@{-}[ld] & & \bullet & & \bullet \ar@{-}[rd] & & \bullet \ar@{-}[ld] \ar@{-}[rd] \\
\bullet & & & & & & \bullet & & \bullet
} \\ \\
L=U^2DDUDDU & L^{(1)}=U^2DUDDUD & L^{(2)}=U^2DDUDDU & L^{(3)}=U^2DUDDUD
\end{array}\]
\caption{Cyclic reorderings of a Dyck path}\label{fig:reorder}
\end{figure}
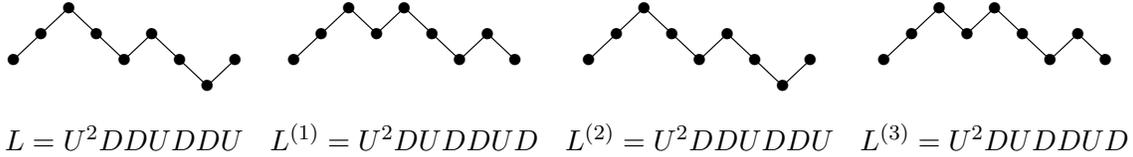
\end{example}

The above example shows that the same lattice path could appear multiple times in the multiset $\{L^{(r)}:0\le r\le n-1\}$.
This issue can be solved in the following way.
Let $I$ be a set of nonnegative integers. 
We define $\defcolor{\L_{I,n,\ell}}$ to be the set of all pairs $(L,i)$ where $L=U^{\ell}DU^{e_1}\cdots DU^{e_n}$ is a lattice path of length $2n$ with $e_1,\ldots,e_n\in I$ and $i\in[\ell]$.
We represent $(L,i)\in\L_{I,n,\ell}$ by marking the $i$th up-step of $L$ with a double line.
We write $\defcolor{\L_{I,n,[n]}}$ for the union of sets $\L_{I,n,\ell}$ with $\ell\in[n]$.
Similarly, we define $\defcolor{\L'_{I,n,\ell}}$ to be the set of all pairs $(L',j)$ where $L'=U^{\ell}DU^{e_1}\cdots DU^{e_n}$ is a Dyck path of length $2n$ with $e_1,\ldots,e_n\in I$ and $j\in[n]$.
We represent $(L',j)\in \L'_{I,n,\ell}$ by marking the $j$th down-step of $L'$ with a double line.
We write $\defcolor{\L'_{I,n,[n]}}$ for the union of sets $\L'_{I,n,\ell}$ with $\ell\in[n]$.

\begin{lemma}\label{lem:Dyck2}
Let $I$ be a set of nonnegative integers.
Then for each $\ell\in[n]$ we have a bijection $\L_{I,n,\ell}\to\L'_{I,n,\ell}$ defined by $(L,i)\mapsto (L^{(r_i(L))},n-r_i(L))$.
\end{lemma}

\begin{proof}
If $(L,i)\in\L_{I,n,\ell}$ then $(L^{(r_i(L))},n-r_i(L)) \in\L'_{I,n,\ell}$ by Lemma~\ref{lem:Dyck1}.
Conversely, suppose that $(L',n-r)\in\L'_{I,n,\ell}$ where $r\in\{0,1,\ldots,n-1\}$. 
Then $L=(L')^{(n-r)}$ satisfies $L^{(r)}=L'$.
By Lemma~\ref{lem:Dyck1} we have $r=r_i(L)$ for a unique $i\in[\ell]$ and thus $(L,i)\in\L_{I,n,\ell}$.
The result follows. 
\end{proof}

\begin{example}
The bijection in Lemma~\ref{lem:Dyck2} is illustrated in Figure~\ref{fig:Dyck2}.
\begin{figure}[h]
\[\begin{array}{cccccc}
\xymatrix @R=8pt @C=8pt @M=-1.5pt { 
& & \bullet \ar@{-}[ld] \ar@{-}[rd] \\
& \bullet \ar@{=}@[red][ld] & & \bullet \ar@{-}[rd] & & \bullet \ar@{-}[ld] \ar@{-}[rd] \\
\bullet & & & & \bullet & & \bullet \ar@{-}[rd] & & \bullet \ar@{-}[ld] \\
& & & & & & & \bullet 
} & \raisebox{-10pt}{$\leftrightarrow$} &
\xymatrix @R=8pt @C=8pt @M=-1.5pt { 
& & \bullet \ar@{-}[ld] \ar@{=}@[red][rd] & & \bullet \ar@{-}[ld] \ar@{-}[rd] \\
& \bullet \ar@{-}[ld] & & \bullet & & \bullet \ar@{-}[rd] & & \bullet \ar@{-}[ld] \ar@{-}[rd] \\
\bullet & & & & & & \bullet & & \bullet
} & \quad
\xymatrix @R=8pt @C=8pt @M=-1.5pt { 
& & \bullet \ar@{=}@[red][ld] \ar@{-}[rd] \\
& \bullet \ar@{-}[ld] & & \bullet \ar@{-}[rd] & & \bullet \ar@{-}[ld] \ar@{-}[rd] \\
\bullet & & & & \bullet & & \bullet \ar@{-}[rd] & & \bullet \ar@{-}[ld] \\
& & & & & & & \bullet 
} & \raisebox{-10pt}{$\leftrightarrow$} 
&\xymatrix @R=8pt @C=8pt @M=-1.5pt { 
& & \bullet \ar@{-}[ld] \ar@{-}[rd] & & \bullet \ar@{-}[ld] \ar@{-}[rd] \\
& \bullet \ar@{-}[ld] & & \bullet & & \bullet \ar@{=}@[red][rd] & & \bullet \ar@{-}[ld] \ar@{-}[rd] \\
\bullet & & & & & & \bullet & & \bullet
} \\ \\
{\color{red}U}UDDUDDU & \leftrightarrow & UU{\color{red}D}UDDUD & \quad U{\color{red}U}DDUDDU & \leftrightarrow & UUDUD{\color{red}D}UD 
\end{array}\]
\caption{The bijecion in Lemma~\ref{lem:Dyck2}}\label{fig:Dyck2}
\end{figure}
\end{example}

Now we may use Dyck paths to give alternate proofs of the main results of Section~\ref{sec:formula}.

\begin{proof}[Another Proof of Proposition~\ref{prop:C_I}]
Let $I$ be a set of nonnegative integers and let $\ell\in[n]$. 
Then
\[ | \L_{I,n,\ell} | = \ell [z^{n-\ell}] \left( \sum_{i\in I} z^i \right)^{n}.\]
By Lemma~\ref{lem:Dyck2}, dividing this number by $n$ gives the result.
\end{proof}

We define $\defcolor{\L_{j,k,n,\ell}}$ to be the set of all lattice paths $U^{\ell}DU^{e_1}\cdots DU^{e_n}$ of length $2n$ with $j$ segments $U^{e_{i_1}},\ldots,U^{e_{i_j}}$ marked, each containing $U^k$, and with one up step $U$ marked in $U^\ell$.
We write $\defcolor{\L_{j,k,n,[n]}}$ for the union of sets $\L_{j,k,n,\ell}$ with $\ell\in[n]$.

\begin{proof}[Another Proof of \eqref{eq:M2}]
Assume $0\le j\le n/k$.
Each element of $\L_{j,k,n+1,1}$ can be constructed in the following way.
First write down $UD$ followed by $2n-jk$ empty spots.
Arbitrarily fill in $n$ of these empty spots with $D$'s and the rest with $U$'s.
Then choose $j$ of the $n+1$ copies of $D$.
Finally, for each of the $j$ chosen $D$'s, insert $U^k$ immediately after it and mark the whole segment of $U$'s containing this $U^k$.
It follows that
\[ \left| \L_{j,k,n+1,1} \right| = {n+1\choose j}{2n-jk\choose n}.\]

We assign a sign $(-1)^j$ to every element of $\L_{j,k,n+1,1}$.
If $L=UDU^{e_0}DU^{e_1}\cdots DU^{e_n}\in \L_{j,k,n+1,1}$ and there exists an $i\in\{0,1,\ldots,n\}$ such that $e_i\ge k$ then let $i$ be as small as possible.
If the segment $U^{e_i}$ is not marked then we mark it; otherwise we unmark it.
This defines a sign-reversing involution on all elements in the union of the sets $\L_{j,k,n+1,1}$ with $0\le j\le n/k$, except those avoiding $U^k$.
Thus, 
\[ \left| \L_{[k-1],n+1,1} \right| = \sum_{0\le j\le n/k} (-1)^j \left| \L_{j,k,n+1,1} \right|. \]
By Lemma~\ref{lem:Dyck2}, dividing this number by $n+1$ gives $M_{k-1,n} = \left| \L'_{[k-1],n,[n]} \right| = \left| \L'_{[k-1],n+1,1} \right|$.
\end{proof}	

\begin{proof}[Another Proof of \eqref{eq:C2}]
Assume $0\le j\le(n-1)/k$.
Every element of $\L_{j,k,n,[n]}$ can be constructed in the following way.
First write down $2n-jk$ empty spots and choose $n+1$ of them.
Fill in the first chosen spot with a marked $U$ and the remaining with $n$ copies of $D$.
Then fill in the rest spots by $U$'s.
Finally, choose $j$ of the $n$ copies of $D$'s and for each of them, insert $U^k$ before it and mark the entire segment of $U$'s containing this $U^k$.
Hence
\[ \left| \L_{j,k,n,[n]} \right|  = {n\choose j} {2n-jk\choose n+1}. \]

We assign $(-1)^j$ to each element of $\L_{j,k,n,[n]}$.
If $L=U^{e_0}DU^{e_1}D\cdots U^{e_n}\in \L_{j,k,n,[n]}$ and there exists an $i\in [n]$ such that $e_i\ge k$ then let $i$ be as small as possible.
If the segment $U^{e_i}$ is not marked then we mark it; otherwise we unmark it.
This defines a sign-reversing involution on all elements in the union of the sets $\L_{j,k,n,[n]}$ with $0\le j\le (n-1)/k$, except those avoiding $DU^k$.
Thus 
\[ \sum_{1\le \ell\le n} \ell \left| \L_{[k-1],n,\ell} \right| = \sum_{0\le j\le (n-1)/k} (-1)^j \left|\L_{j,k,n,[n]} \right|. \]
Dividing this number by $n$ and using Lemma~\ref{lem:Dyck2} we have the formula \eqref{eq:C2} for $C_{k,n}$.
\end{proof}

\begin{proof}[Another Proof of Proposition~\ref{prop:FormulaLargest}]
If $L=U^{e_0}DU^{e_1}\cdots DU^{e_n}$ be a lattice path of length $2n$ such that $k$ divides $e_1,\ldots,e_n$,
then $e_1+\cdots+e_n=jk$ for some nonnegative integer $j\le n/k$ and $\ell=n-jk$.
If $I=\{0,k,2k,\ldots\}$ then
\[ \left| \L_{I,n,n-jk} \right| = (n-jk){n+j-1 \choose j}.\]
Hence the result follows from Lemma~\ref{lem:Dyck2}.
\end{proof}

\section{Some refinements}\label{sec:refinement}

The Catalan number $C_n$ can be refined to the \emph{Narayana number} 
\[ N_{n,r} := \frac 1n {n\choose r}{n\choose r-1} \]
which enumerates plane trees with $n+1$ total nodes, of which $r$ are internal, or equivalently, Dyck paths of length $2n$ with $r$ local maxima, called \demph{peaks}.
We present similar refinements for $M_{k,n}$ and $C_{k,n}$.

Assume $k\ge1$ and $0\le r\le n$.
We define $\defcolor{M_{k,n,r}}$ to be the number of plane trees with $n+1$ nodes, of which $r$ are internal, such that the degree of each node is no more than $k$.
Similarly, $\defcolor{C_{k,n,r}}$ denotes the number of plane trees with $n+1$ nodes, of which $r$ are internal, such that each non-root node has degree less than $k$.
We have $M_{k,0,0} = C_{k,0,0} = 1$ and $M_{k,n,0}=C_{k,n,0}=0$ for $n\ge1$.
Moreover, if $n\ge1$ then 
\[ M_{k,n} = M_{k,n,1} + \cdots + M_{k,n,n} \quad\text{and}\quad 
C_{k,n} = C_{k,n,1} + \cdots + C_{k,n,n}. \]

To derive formulas for $M_{k,n,r}$ and $C_{k,n,r}$, let $I$ be a set of strictly positive integers and denote by $\defcolor{C_{I,n,r}^{(\ell)}}$ the number of plane trees whose multi-degree $(\md_0,\ldots,\md_n)$ satisfies $\md_0=\ell$, $\md_1,\ldots,\md_n\in I\cup\{0\}$, and $|\{i\in\{0,\ldots,n\}: \md_i>0\}| = r$.
We have
\begin{equation}\label{eq:r}
M_{k,n,r} = C^{(1)}_{[k],n+1,r+1} \quad\text{and}\quad
C_{k+1,n,r} = \sum_{\ell\ge0} C_{[k],n,r}^{(\ell)}. 
\end{equation}
\begin{proposition}\label{prop:C_nr}
Let $I$ be a set of positive integers. 
For $\ell\ge0$ and $r\in\{0,\ldots,n-1\}$ we have 
\[ C_{I,n,r+1}^{(\ell)} = \frac\ell{n} \left[x^{r} z^{n-\ell}\right] \left( 1+x\sum_{i\in I} z^i \right)^{n} 
= \frac{\ell}{n}\sum_{\substack{ \sum_{i\in I}m_i = r \\ \sum_{i\in I} im_i = n-\ell} } {n\choose r} {r\choose m_i: i\in I }.\]
\end{proposition}

\begin{proof}
This result follows from Proposition~\ref{prop:C^l}.
One can also prove it in a similar way as the proof of Proposition~\ref{prop:C_I} provided in Section~\ref{sec:Dyck} using Dyck paths and Lemma~\ref{lem:Dyck2}.
\end{proof}

\begin{remark}
When $I$ consists of all positive integers we have
\[ C_{I,n,r+1}^{(\ell)} = \frac\ell{n} \left[x^{r} z^{n-\ell}\right] \left( 1+\frac{xz}{1-z} \right)^{n} 
= \frac\ell n {n\choose r} [z^{n-r-\ell}] (1-z)^{-r} 
= \frac\ell n {n\choose r} {n-\ell-1 \choose r-1} \]
and $C_{I,n+1,r+1}^{(1)}$ equals the {Narayana number} $N_{n,r}$.
\end{remark}


\begin{proposition}
For $k\ge1$ and $n\ge0$ we have
\[ M_{k,n,r} 
= \frac1{n+1} \sum_{\substack{ \lambda\subseteq k^{n+1} \\ |\lambda|=n\\ \ell(\lambda)=r }} m_\lambda(1^{n+1}).  \]
For $k\ge0$ and $n\ge1$ we have
\[ C_{k+1,n,r} 
= \sum_{\substack{\lambda\subseteq k^n\\ |\lambda|<n\\ \ell(\lambda)=r }} \frac{n-|\lambda|}{n} m_\lambda(1^n). \]
\end{proposition}

\begin{proof}
The result follows from \eqref{eq:r} and Proposition~\ref{prop:C_nr}.
\end{proof}

\begin{proposition}
For $k\ge1$ and $n\ge0$ we have
\[ M_{k,n,r} = \frac{1}{n+1} {n+1\choose r} \sum_{0\le j\le (n-r)/k} (-1)^j {r\choose j}{n-jk-1\choose r-1}.  \]
For $k\ge1$ and $n\ge1$ we have
\[ C_{k+1,n,r} = \sum_{0\le j\le (n-r)/k} \frac{(-1)^j}n {n\choose r-1}  {r-1\choose j} {n-jk\choose r}. \]
\end{proposition}

\begin{proof}
The proof is similar to the proof of Theorem~\ref{thm:FormulaAlt}.
By Proposition~\ref{prop:C_nr} we have
\begin{eqnarray*}
C_{[k],n,r+1}^{(\ell)} &=& 
\frac\ell{n} \left[x^r z^{n-\ell}\right] \left( 1+\frac{xz(1-z^k)}{1-z} \right)^{n} \\
&=& \frac\ell n {n\choose r} [z^{n-\ell-r}] \left( \frac{1-z^k}{1-z}\right)^{r} \\
&=& \frac\ell n {n\choose r} \sum_{0\le j\le (n-\ell-r)/k} (-1)^j {r\choose j} {n-\ell-jk-1 \choose r-1}.
\end{eqnarray*}
This implies the desired formula for $M_{k,n,r} = C_{[k],n+1,r+1}^{(1)}$. 
Since $C_{k+1,n,r} = \sum_{\ell\ge0} C_{[k],n,r}^{(\ell)}$, we have
\begin{eqnarray*}
C_{k+1,n,r} 
&=& \sum_{1\le \ell\le n} \frac{\ell}{n} {n\choose r-1} \sum_{0\le j\le (n-\ell-r+1)/k} (-1)^j {r-1\choose j}{n-\ell-1-jk\choose r-2} \\
&=&  \sum_{0\le j\le (n-r)/k} \frac{(-1)^j}n {n\choose r-1}  {r-1\choose j} \sum_{1\le \ell\le n-r+1-jk} \ell {n-\ell-1-jk\choose r-2}  \\
&=&  \sum_{0\le j\le (n-r)/k} \frac{(-1)^j}n {n\choose r-1}  {r-1\choose j} {n-jk\choose r}. 
\end{eqnarray*}
Here the last step follows from \eqref{eq:second}.
\end{proof}

\begin{proof}[Another Proof]
We first observe the number of peaks of a Dyck path $L$ of length $2n$ to be one greater than the number of local minimum points of $L$ other than $(0,0)$ and $(2n,0)$.
Call these local minimum points \demph{valleys}.
Note this relation does not hold for the numbers of peaks and valleys of a general lattice path.
However, the operation $L\mapsto L^{(i)}$ preserves the number of valleys.

Now recall the definition for $\L_{j,k,n,\ell}$ and $\L_{j,k,n,[n]}$ from Section~\ref{sec:Dyck}.
Write $\L_{j,k,n,\ell,r}$ and $\L_{j,k,n,[n],r}$ for the maximal subsets of $\L_{j,k,n,\ell}$ and $\L_{j,k,n,[n]}$, respectively, whose members have exactly $r-1$ valleys.
We will constructively prove the following formulas:
\[ (i)\rule{13pt}{0pt}|\L_{j,k+1,n+1,1,r+1}| = {n+1\choose r} {r\choose j}{n-1-jk\choose r-1}\,,\]
\[ (ii)\rule{11pt}{0pt}|\L_{j,k+1,n,[n],r}| = {n\choose r-1}  {r-1\choose j} {n-jk\choose r}\,.\rule{20pt}{0pt} \]
 
\emph{Construction} $(i)$: Every element of $\L_{j,k+1,n+1,1,r+1}$ can be constructed in the following way.
First write down a $U$ followed by $n+1$ copies of $D$. 
Then insert an up-step immediately after $r$ of these down-steps so that there are $r$ valleys.
Next choose $j$ of these $r$ valleys and insert $U^k$ right after each of them.
Finally partitioning $n-r-jk$ up-steps into $r$ possibly empty blocks and insert them immediately after the $r$ valleys.\smallskip

\emph{Construction} $(ii)$: Every element of $\L_{j,k,n,[n],r}$ can be constructed in the following way.
First write $n-jk$ up-steps and mark $r$ of them.
This partitions the $n-jk$ up-steps into $r+1$ blocks: the first one ends right before the first marked $U$, the second one starts from the first marked $U$ and ends right before the second marked $U$, and so on.
All these blocks of $U$'s are nonempty except possibly the first one.
Then choose $r-1$ of $n$ down-steps, put the first and second blocks of $U$'s before the first chosen $U$, the third block of $U$'s after the first chosen $U$, the fourth block after the second chosen $U$, and so on.
Leave the first and second blocks of $U$'s alone so that they still contain a marked $U$.
Finally, choose $j$ of the remaining $r-1$ blocks, append $U^{k}$ to each of them, and mark these $j$ expanded blocks.

The rest of the proof is similar to the proofs for \eqref{eq:M2} and \eqref{eq:C2} in Section~\ref{sec:Dyck}.
\end{proof}

\section{Remarks and Questions}\label{sec:questions}

\subsection{}
Let $*$ be a binary operation defined on a set $X$ and let $x_0,\ldots,x_n$ be $X$-valued indeterminates.
Denote by $C_{*,n}$ the number of distinct functions from $X^{n+1}$ to $X$ obtained by parenthesizing $x_0*\cdots *x_n$. 
In general $1\le C_{*,n}\le C_n$, and if $*$ is $k$-associative then $1\le C_{*,n}\le C_{k,n}$.
For $n\ge2$ one has $C_{*,n}=1$ if and only if $*$ is associative.
If $*$ is the $k$-associative operation defined in Example~\ref{example:ring}, then $C_{*,n}=C_{k,n}$ by Proposition~\ref{prop:binary2} (iv).
Can one characterize when $C_{*,n}=C_n$ and when $C_{*,n}=C_{k,n}$ for $k>1$? 
Do other interesting numbers $C_{*,n}$ arise from binary operations we have not yet considered?

\subsection{}
We have seen the modular Catalan number $C_{k,n}$ enumerates several restricted families of Catalan objects.
There are many other families of Catalan objects, such as those presented in~\cite{StanleyCatalan}.
Elementary connections between some of these objects and those studied here lead to other restricted families of Catalan object enumerated by $C_{k,n}$.
For example, there is a bijection between $2\times n$ tableaux and chains in the Bruhat order of the Grassmannian of $2$-planes in $(n{+}2)$-space which we did not discuss.
It may be interesting to extend our investigation by exploring some less-elementary connections between Catalan objects.

\subsection{}
We have seen that the poset of $\T_{k,n}$ consisting of all binary trees avoiding $\comb_k^{1}$ under the Tamari order is the same as the Tamari lattice $\T_n$ when $k\ge n$ and is isomorphic to the Boolean lattice $\mathcal B_{n-1}$ when $k=2$. What can be said about this poset when $2<k<n$?

\subsection{}
An exercise shows $C_{n+1}/C_n \to 4$ as $n\rightarrow\infty$.
One may compare this to the asymptotic behavior of the $k$-modular analogue $C_{k,n+1}/C_{k,n}$.
There is not much to compare for $k=2$ as $C_{2,n+1}/C_{2,n}=2$ for all $n\ge1$.
Computer experimentation suggests $\lim_{n\to\infty} C_{3,n+1}/C_{3,n} = 3$, $3<\lim_{n\to\infty} C_{k,n+1}/C_{k,n} <4$ for $k\ge4$, and 
\[\lim_{n\to\infty} \frac{M_{k-1,n+1}}{M_{k-1,n}}=\lim_{n\to\infty}\frac{C_{k,n+1}}{C_{k,n}}\,.\]
Which, if any, of these suggestions are true?

\subsection{}
The sequence $C_{3,n}$ is the OEIS sequence A005773, which enumerates various objects. 
Taking $k=3$ in \eqref{eq:C1} and assuming $j$ is the number of $1$'s in $\lambda$, one obtains
\begin{eqnarray*}
C_{3,n} &=& \sum_{1\le\ell\le n}\frac{\ell}{n} \sum_{0\le j\le n-\ell} {n\choose \frac{n+\ell-j}2, j, \frac{n-\ell-j}2 } \\
&=& \sum_{0\le j\le n-1} \sum_{1\le\ell\le n-j} {n-1\choose j} {n-j\choose \frac{n-\ell-j}2} \frac\ell{n-j} \\
&=& \sum_{0\le i\le n-1} {n-1\choose i} \sum_{1\le\ell\le i+1} {i+1\choose \frac{i+1-\ell}2} \frac\ell{i+1} \\
&=& \sum_{0\le i\le n-1} {n-1\choose i} \sum_{0\le r\le i/2} \left[ {i \choose r} - {i\choose r-1} \right] \\
&=& \sum_{0\le i\le n-1} {n-1\choose i} {i\choose \lfloor i/2 \rfloor}.
\end{eqnarray*}
Here we assume ${n\choose m}:=0$ whenever $m$ is not a nonnegative integer. This formula for $C_{3,n}$ 
was obtained by Gouyou-Beauchamps and Viennot~\cite{DirectedAnimal} during their study of directed animals, and also obtained by Panyushev~\cite{Panyushev} using the affine Weyl group of the Lie algebra $\mathfrak{sp}_{2n}$ or $\mathfrak{so}_{2n+1}$.
We do not currently have an understanding of how these objects are related to the objects in Proposition~\ref{prop:C counts}.
Can one generalize the above formula for $C_{3,n}$ to $C_{k,n}$?

 
\subsection{}
Is each $k$-connected component of $\T_n$ (under the $k$-associative order) a meet-semilattice?
Is every interval in a $k$-connected component of $\T_n$ a lattice?

\end{document}